\documentclass[a4paper,11pt,oneside]{article}
\usepackage{graphicx}
\usepackage{amscd}
\usepackage{amsmath}
\usepackage{caption}
\usepackage{amsfonts}
\usepackage{amssymb}
\usepackage{amsthm}
\usepackage{mathrsfs}
\usepackage{multicol}
\usepackage{color}
\usepackage[english]{babel}
\usepackage[T1]{fontenc}
\usepackage{textcomp}
\usepackage[utf8]{inputenc}
\usepackage{enumitem}
\usepackage{indentfirst}
\usepackage[pagewise]{lineno}

\newcommand{\RR}{{{\rm I} \kern -.15em {\rm R} }}
\begin{document}
	\theoremstyle{plain} \newtheorem{thm}{Theorem}[section] \newtheorem{cor}[thm]{Corollary} \newtheorem{lem}[thm]{Lemma} \newtheorem{prop}[thm]{Proposition} \theoremstyle{definition} \newtheorem{defn}{Definition}[section] \theoremstyle{remark} \newtheorem{oss}{Remark}[section]
	\newtheorem{ex}{Example}[section]
	\newtheorem{lemma}{Lemma}[section]
	\title{Semiconcavity for the minimum time problem \\in presence of time delay effects}
	\author{Elisa Continelli\footnote{Email: elisa.continelli@graduate.univaq.it}\hspace{0.3cm}\&\hspace{0.2cm}Cristina Pignotti\footnote{Email: cristina.pignotti@univaq.it} \\Dipartimento di Ingegneria e Scienze dell'Informazione e Matematica\\
		Universit\`{a} degli Studi di L'Aquila\\
		Via Vetoio, Loc. Coppito, 67100 L'Aquila Italy}

	\maketitle
	\begin{abstract}
		In this paper, we deal with a minimum time problem in presence of a time delay $\tau.$ The value function of the optimal control problem is no longer defined in a subset of $\mathbb{R}^{n}$, as it happens in the undelayed case, but its domain is a subset of the Banach space $C([-\tau,0];\mathbb{R}^{n})$. For the undelayed minimum time problem, it is known that the value function associated with it is semiconcave in a subset of the reachable set and is a viscosity solution of a suitable Hamilton-Jacobi-Bellman equation. The Hamilton-Jacobi theory for optimal control problems involving time delays has been developed by several authors. Here, we are rather interested in investigating the regularity properties of the minimum time functional. Extending classical arguments, we are able to prove that the minimum time functional is semiconcave in a suitable subset of the reachable set. 
	\end{abstract}
		\providecommand{\keywords}[1]{\textbf{Keywords:} #1}
	\providecommand{\mathsub}[1]{\textbf{2020 Mathematics Subject Classification:} #1}
	\\\\\mathsub{ 34K05, 49K15, 49L20}	
	\vspace{5 mm}

	\section{Introduction}	
	Consider the nonlinear system
	\begin{equation}\label{nodel}
		\begin{cases}
			y'(t)=f(y(t),u(t)),\quad &\,  t\geq 0,\\y(0)=x\in \mathbb{R}^{n},
		\end{cases}
	\end{equation}
where $U$ is a compact subset of $\mathbb{R}^{m}$ and $f:\mathbb{R}^{n}\times U\rightarrow\mathbb{R}^{n}$ is a given function. A measurable function $u:[0,+\infty)\rightarrow U$ is called a {\em control} and the solution to the state equation \eqref{nodel} with the initial datum $x$ corresponding to the control $u$ is denoted by $y(\cdot;x,u)$. Given a nonempty closed set $\mathcal{K}\subseteq\mathbb{R}^{n}$, called the {\em target}, the {\em reachable set} $\mathcal{R}$ is the set of all points $x\in \mathbb{R}^{n}$ that can be steered to $\mathcal{K}$ in finite time. The {\em minimum time problem} for the system \eqref{nodel} is the following:
\begin{equation}\label{mtpnodel}
	\text{minimize } \theta(x,u), \text{ over all controls } u,
\end{equation}
where $ \theta(x,u)$ is the first time at which the trajectory starting at $x\in \mathbb{R}^n$ with control $u$ reaches the target $\mathcal K,$ i.e.

$$
 \theta(x,u):=\min\{t\geq 0:y(t,x,u)\in \mathcal{K}\}.$$
The value function associated to the optimal control problem \eqref{mtpnodel} is called the {\em minimum time function} and it is the function $T:\mathcal{R}\rightarrow\mathbb{R}$ defined as
\begin{equation}\label{mtfnodel}
	T(x):=\inf_{u} \theta(x,u),\quad \forall x\in \mathcal{R}.
\end{equation}
\\The minimum time problem \eqref{mtpnodel} is a classical topic in control theory that has caught the attention of many researchers in the past decades. In particular, the regularity properties of the minimum time function $T$ given by \eqref{mtfnodel} have been largely investigated. It is known that, under a controllability assumption known as the Petrov condition (see \cite{Petrov1, Petrov2}), the minimum time function is locally Lipschitz continuous in the reachable set. However, there are simple examples showing that the minimum time function fails to be differentiable everywhere. In \cite{Cansin}, it has been proved that, under suitable assumptions, the minimum time function is semiconcave in $\mathcal{R}\setminus \mathcal{K}.$ The semiconcavity result has been then  extended to general exit time problems in \cite{CPS}.  Moreover, it has been proved in \cite{Bardi} that the minimum time function is the unique viscosity solution of a suitable Hamilton-Jacobi-Bellman equation. Semiconcavity is a very useful property to get uniqueness results for weak solutions of Hamilton-Jacobi equations. Also, the semiconcavity property allows deriving some optimality conditions (see e.g. \cite{CFS} for the minimum time function, \cite{Luong, pignotti} for exit time problems). For more details on semiconcave functions, Hamilton-Jacobi theory and viscosity solutions we refer to \cite{Cannarsa_libro, Capuzzo, CL,CEL}. For other results regarding semiconcavity of different types of equations see for instance \cite{Albano, ABC, ACS, Bonnet, Han}.

In this paper, we deal with the minimum time problems for the delayed control system
\begin{equation}\label{delcs}
	\begin{cases}
		y'(t)=f(y(t-\tau),u(t)),\quad &\,  t\geq 0,\\y(s)=x(s),\quad &s\in [-\tau,0],
	\end{cases}
\end{equation}
where $x\in C([-\tau,0];\mathbb{R}^{n})$, being $C([-\tau,0];\mathbb{R}^{n})$ the Banach space of all continuous functions defined in $[-\tau,0]$ with values in $\mathbb{R}^{n}$. For a complete understanding of differential equations of the type \eqref{delcs}, we refer the reader to \cite{Halanay, HL}. Due to the fact that the initial data are functions whose domain is the time interval $[-\tau,0]$, the minimum time function associated to the minimum time problem for system \eqref{delcs} is no longer defined in a subset of $\mathbb{R}^{n}$. Indeed, the reachable set $\mathcal{R}$ is now a subset of the inifinite dimensional space $C([-\tau,0];\mathbb{R}^{n})$. 

The Hamilton-Jacobi theory for optimal control problems involving time delays has been developed by several authors (\cite{Barron, Boccia, Bonalli, Luk1, Luk2, Luk3, Plaksin1, Plaksin2, Soner, Vinter, Zhu}). In all of these works the optimal control problems under consideration are with finite time horizon. Also, in \cite{Luk1, Luk2, Luk3, Plaksin1, Plaksin2, Zhu}, optimal control problems associated to control systems with dynamics depending not just on a past instant but on a whole interval of the past history of the trajectory (in the literature one refers to such systems as hereditary systems) are investigated and a notion of differentiability for functionals belonging to the Banach space $C([-\tau,0];\mathbb{R}^{n})$, the so-called coinvariant (or ci)-differentiability, is exploited. In \cite{Barron, Boccia, Vinter}, optimality conditions, with particular attention to the Pontryagin maximum principle, are given for optimal control problems in presence of time delays. We also mention \cite{Crandall1, Crandall2, Crandall3} for a general study of Hamilton-Jacobi equations in infinite dimension and \cite{Ishii} for the analysis of a class of Hamilton-Jacobi equations in Hilbert spaces.

Motivated by this, we analyze the regularity properties of the minimum time function associated with the delayed control system \eqref{delcs}. Extending the arguments of \cite{Cansin}, we are able to prove that, also in this functional setting, the Petrov controllability condition implies the local Lipschitz continuity of the minimum time function. Furthermore, we show that the minimum time function is semiconcave in a suitable subset of the reachable set. {Both the local Lipschitz continuity and the semiconcavity of the minimum time function will be proven assuming a smallness condition on the time delay size.}

{Let us note that, with respect to the  arguments for the undelayed case (cf. \cite{Cansin}), since the minimum time is a functional mapping functions $x\in C([-\tau,0];\mathbb{R}^{n})$ into $[0,+\infty],$ additional properties have to be taken into account and a more delicate analysis is required.  In particular, dealing with the system's trajectories requires distinguishing different time ranges, the interval $[0,\tau]$ for which the delayed time belongs to the negative interval $[-\tau, 0],$ and so the delayed state concerns the initial datum, and the successive times, $t>\tau.$ Consequently, different situations have to be considered and new properties, with respect to the case without any time delays, are involved, such as the Lipschitz continuity of the initial data and the smallness assumption on the time delay $\tau.$ }

{The paper is organized as follows. In Sect. \ref{Prel} we present some preliminary definitions and notations and we introduce rigorously the minimum time problem for systems involving time delays. In Sect. 3, we show that the minimum time function is locally Lipschitz continuous in suitable subsets of the reachable set, namely subsets of the reachable set consisting of Lipschitz continuous initial data, provided that the time delay $\tau$ satisfies a suitable smallness condition. In Sect. 4, again under a smallness condition on the time delay size, we prove our main result, which ensures that the minimum time function is semiconcave in a suitable subset of $\mathcal{R}_{M}$, where $\mathcal{R}_{M}$ is the set of all functions $x$ in the reachable set $\mathcal{R}$ that are Lipschitz continuous of constant $M$, being $M$ the constant of boundedness of the dynamics $f$. Finally, in Sect. 5 we make some conclusions and comments about the analysis carried out throughout this paper.}
	\section{Preliminaries}\label{Prel}

	Let $\lvert\cdot\rvert$ be the usual norm on $\mathbb{R}^{n}$. Given a nonempty closed set $\mathcal{K}\subseteq\mathbb{R}^{n}$, we denote with $d_{\mathcal{K}}(\cdot)$ the distance function from $\mathcal{K}$, namely

	$$d_{\mathcal{K}}(z)=\inf_{y\in \mathcal{K}}\lvert y-z\rvert,\quad \forall z\in \mathbb{R}^{n}.$$
Moreover, for any $\rho>0$, we set

 $$\mathcal{K}_{\rho}=\{z\in\mathbb{R}^{n}:d_{\mathcal{K}}(z)<\rho \}.$$  
	We shall denote with $C([-\tau,0];\mathbb{R}^{n})$ the Banach space of all continuous functions defined in  $[-\tau,0]$ with values in $\mathbb{R}^{n}$. We endow $C([-\tau,0];\mathbb{R}^{n})$ with the uniform norm

	$$\lVert x\rVert_{\infty}=\sup_{s\in[-\tau,0]}\lvert x(s)\rvert.$$
	We recall the definition of a semiconcave function.
	\begin{defn}
		A continuous function $v:\Omega\rightarrow\mathbb{R}$, with $\Omega\subseteq \mathbb{R}^{n}$, is called \textit{semiconcave} if, for any convex set $K\subset\subset \Omega$, there exists $c_{K}>0$ such that

		$$v(x+h)+v(x-h)-2v(x)\leq c_{k}\lvert h\rvert^{2},$$
		for any $x,h$ such that $x,x+h,x-h\in K$.
	\end{defn}
	For the analysis we carry out in this paper, we need to extend the above definition to functionals whose domain is a subset of the Banach space $C([-\tau,0];\mathbb{R}^{n})$.
	\begin{defn}
		A continuous functional $v:S\rightarrow\mathbb{R}$, with $S\subseteq C([-\tau,0];\mathbb{R}^{n})$, is called \textit{semiconcave} if, for any convex set $K\subset\subset S$, there exists $c_{K}>0$ such that

		$$v(x+h)+v(x-h)-2v(x)\leq c_{k}\lVert h\rVert_{\infty}^{2},$$
		for any $x,h$ such that $x,x+h,x-h\in K$.
	\end{defn}
	\subsection{The minimum time problem with time delay}

Let $\mathcal{K}$ be a nonempty closed subset of $\mathbb{R}^{n}$ and $U$ be a subset of $\mathbb{R}^{m}$. Let $f:\mathbb{R}^{n}\times U\rightarrow\mathbb{R}^{n}$ be a given function. 
\\Consider the delayed control system
\begin{equation}\label{CS}
\begin{cases}
		y'(t)=f(y(t-\tau),u(t)),\quad &\,  t\geq 0,\\y(s)=x(s),\quad &s\in [-\tau,0],
	\end{cases}
\end{equation}
	where $x\in C([-\tau,0];\mathbb{R}^{n})$. 
{We emphasize that, due to the presence of the time delay, the initial datum is here a function; then the state equation \eqref{CS} is now a functional differential equation. Therefore, the analysis, even if based on the classical theory, requires more careful and finer arguments.}

In this paper, the set $\mathcal{K}$ plays the role of the \textit{target} to which one wants to steer the trajectories of the nonlinear control system \eqref{CS}. We assume that the following conditions hold.
	\begin{itemize}
		\item[$(H_{1})$] $f:\mathbb{R}^{n}\times U\rightarrow\mathbb{R}^{n}$ is continuous and satisfies
		$$\lvert f(z,u)\rvert\leq M,\quad \forall z\in \mathbb{R}^{n}, u\in U, $$$$\lvert f(z,u)-f(y,u)\rvert\leq L \lvert z-y\rvert,\quad \forall z,y\in \mathbb{R}^{n}, u\in U.$$
		\item[$(H_{2})$] $U\subseteq \mathbb{R}^{m}$, $\mathcal{K}\subseteq \mathbb{R}^{n}$ are compact sets.
		{\item[$(H_{3})$] The gradient of $f$ with respect to the space variable $x$ exists everywhere and is locally Lipschitz continuous in $x$, uniformly in $u$. }
	\end{itemize}
A control for equation \eqref{CS} is a measurable function $u:[0,\infty)\rightarrow U$.

Given a control $u:[0,\infty)\rightarrow U$ and an initial datum $x\in C([-\tau,0];\mathbb{R}^{n})$, adopting a step-by-step procedure and using the assumption $(H_{1})$, we get that there exists a unique solution to the control system \eqref{CS} associated to the control $u$ for the initial condition $x$, denoted by $y(\cdot;x,u)$. For further details about the existence and uniqueness of functional differential equations, we refer to \cite{Halanay}, \cite{HL}.

Given a solution $y(\cdot;x,u)$ of the control system \eqref{CS}, for every $t\geq 0$, let $y_{t}(\cdot;x,u)$ be the function in $ C([-\tau,0];\mathbb{R}^{n})$ given by

$$y_{t}(s;x,u)=y(t+s;x,u),\quad \forall s \in [-\tau,0].$$
For any initial datum $x\in C([-\tau,0];\mathbb{R}^{n})$ and for any control $u$, we set 

$$\theta(x,u)=\min\{t\geq0:y(t;x,u)\in \mathcal{K} \},$$
i.e. $\theta(x,u)\in [0,+\infty]$ is the first time at which the trajectory $y(\cdot;x,u)$ reaches the target $\mathcal{K}$.
Moreover, we define the \textit{reachable set} $\mathcal{R}$ as the set of all initial data starting from which the target can be reached in a finite time, namely 

$$\mathcal{R}=\{x\in C([-\tau,0];\mathbb{R}^{n}):\theta(x,u)<+\infty, \,\text{for some control}\, u \}.$$
The optimal control problem we deal with, called the \textit{minimum time problem}, is the following: 
\begin{equation}\label{MTP}
\quad \text{minimize } \theta(x,u), \text{ over all controls } u.
\end{equation}
The value function associated to \eqref{MTP} is called \textit{minimum time function}.
{Since the reachable set is a subset of the Banach space 
$ C([-\tau,0];\mathbb{R}^{n}),$ the minimum time function is no longer a function defined  on a subset of  $\mathbb{R}^{n}$, as in the undelayed case, but a functional,} 
$T:\mathcal{R}\rightarrow [0,+\infty),$ defined as

$$T(x)=\inf_{u}\theta(x,u),\quad \forall x\in\mathcal{R}.$$
Our interest is in investigating the regularity properties of the minimum time functional over a certain subset of the reachable set. For this purpose, we will assume the following controllability condition of Petrov type:
\begin{itemize}
	\item[$(H_{4})$] There exist $\mu,\sigma>0$ such that, for any $x\in\mathbb{R}^{n}$ with $x\in \mathcal{K}_{\sigma}\setminus\mathcal{K}$, there exists $u_{x}\in U$ such that
	$$f(x,u_{x})\cdot\frac{x-\pi(x)}{\lvert x-\pi(x)\rvert}\leq -\mu,$$
	for some $\pi(x)\in \mathcal{K}$ such that $d_{\mathcal{K}}(x)=\lvert x-\pi(x)\rvert.$
\end{itemize}	
\section{Local Lipschitz continuity of the value function}
Now, we investigate the Lipschitz continuity of the minimum time functional. We first present a version of the dynamic programming principle that is suitable for our functional setting.
\begin{prop}[Dynamic programming principle]
	\label{din}
	Let $x\in \mathcal{R}$. Then, for every control $u$ such that $\theta(x,u)<+\infty$, it holds
	\begin{equation}\label{dpp}
		T(x)\leq t+T(y_{t}(\cdot;x,u)), \quad\forall {t\in [0,\theta(x,u)]}.
	\end{equation}
\end{prop}	
\begin{proof}
	{By contradiction, we assume that there exist a control $\tilde{u}$ satisfying $\theta(x,\tilde{u})<+\infty$ and  a time $\tilde{t}\in [0,\theta(x,\tilde u)]$ such that }

$${T(x)>\tilde{t}+T(y_{\tilde{t}}(\cdot;x,\tilde{u})).}$$
	{Then, there exists $\epsilon>0$ small enough to have }
	$${T(x) \geq \tilde{t}+T(y_{\tilde{t}}(\cdot;x,\tilde{u}))+\epsilon.}$$
	{Now, we define the control $\bar{u}:[0,+\infty)\rightarrow U$ given by}
	
$${\bar{u}(t)=\begin{cases}
		\tilde{u}(t),\quad t\in [0,\tilde{t}],\\u(t),\quad t>\tilde{t},
	\end{cases}}$$
	{where $u$ is a control such that $T(y_{\tilde{t}}(\cdot;x,\tilde{u}))+\epsilon>\theta(y_{\tilde{t}}(\cdot;x,\tilde{u}),u)$. Then, by definition of $\bar{u}$ we get}
	
$${\theta(x,\bar{u})=\tilde{t}+\theta(y_{\tilde{t}}(\cdot;x,\tilde{u}),u)<\tilde{t}+\epsilon+ T(y_{\tilde{t}}(\cdot;x,\tilde{u}))\leq T(x),}$$
{which contradicts the optimality of $T(x)$. Thus, \eqref{dpp} is valid.}
\end{proof}
For our purpose, given $\bar{M}>0$, we set 

$$\text{Lip}_{\bar{M}}([-\tau,0];\mathbb{R}^{n})=\{x\in C([-\tau,0];\mathbb{R}^{n}):\lvert x(s)-x(t)\rvert\leq \bar{M} \lvert s-t\rvert,\,\forall s,t\in [-\tau,0]\},$$
$$\mathcal{R}_{\bar{M}}=\{x\in \mathcal{R}:x\in \text{Lip}_{\bar{M}}([-\tau,0];\mathbb{R}^{n})\}.$$
{Our goal is to prove that the minimum time functional is locally Lipschitz continuous in suitable subsets of the reachable set consisting of Lipschitz continuous initial data. }To this aim, we need the following result ensuring that the Petrov condition $(H_{4})$ implies an estimate of the minimum time functional in terms of the distance from the target.
\begin{prop}\label{stima_distanza}
 Assume  $(H_{1})$, $(H_{2})$, $(H_{4})$ hold
and let $\bar M>0$. Let us denote $M^*:=\max\{M,\bar{M}\}$ and
	assume that $\displaystyle{\tau<\frac{\mu}{2 M^*L}}$. Then, there exist $\delta,C>0$ such that
	\begin{equation}\label{estimate}
		T(x)\leq C d_{\mathcal{K}}(x(0)), 
	\end{equation}
	for every {$x\in \text{Lip}_{M^*}([-\tau,0];\mathbb{R}^{n})$ with $x(0)\in \mathcal{K}_{\delta}$}.
\end{prop}
	\begin{proof}
		From our assumptions, possibly taking a larger $M$, we may assume $M\geq\mu$. {Also, it is not restrictive to assume $\bar{M}\geq M$. Thus, $M^*=\bar{M}$ and the smallness condition on the time delay size reads as $\tau<\frac{\mu}{2\bar{M}L}$.} Let $x_{0}\in \text{Lip}_{\bar{M}}([-\tau,0];\mathbb{R}^{n})$ be such that $x_{0}(0)\in \mathcal{K}_{\delta}$, where 

$$\delta=\min\left\{\frac{M}{L},\sigma\right\}.$$
		We can suppose that $x_{0}(0)\in \mathcal{K}_{\delta}\setminus\mathcal{K}$, since otherwise inequality \eqref{estimate} is trivially satisfied for any positive constant $C$. 
		
Following \cite{Cansin}, we define inductively a sequence $(x_{j})_{j}\subset C([-\tau,0];\mathbb{R}^{n})$ such that $x_{j}(0)\in  \mathcal{K}_{\delta}\setminus\mathcal{K}$, for all $j\geq 0$, in the following way. For all $j\geq 0$, we set \begin{equation}\label{xj}
			x_{j+1}=y_{t_{j}}(\cdot;x_{j},u_{x_{j}}),
		\end{equation}
		where
		\begin{equation}\label{tj}
			t_{j}=\frac{\mu-2\bar ML\tau}{4M^{2}}d_{\mathcal{K}}(x_{j}(0)),
		\end{equation}
		and $u_{x_{j}}$ is chosen according to $(H_{3})$ related to $x_j(0),$ i.e. 

$$f(x_{j}(0),u_{x_{j}})\cdot\frac{x_{j}(0)-\pi(x_{j}(0))}{\lvert x_{j}(0)-\pi(x_{j}(0))\rvert}\leq -\mu.$$
First of all, we have that $x_{j}\in \text{Lip}_{\bar{M}}([-\tau,0];\mathbb{R}^{n})$, for all $j\geq 0$. Indeed, by induction, if $j=0$ we know that $x_{0}\in \text{Lip}_{\bar{M}}([-\tau,0];\mathbb{R}^{n})$. Now, we assume that $x_{j}\in \text{Lip}_{\bar{M}}([-\tau,0];\mathbb{R}^{n})$, for some $j\geq 0$, and we show that also $x_{j+1}\in \text{Lip}_{\bar{M}}([-\tau,0];\mathbb{R}^{n})$. For every $s,t\in [-\tau,0]$, by definition of $x_{j+1}$, it holds 

$$\lvert x_{j+1}(s)-x_{j+1}(t)\rvert=\lvert y(t_{j}+s;x_{j},u_{x_{j}})-y(t_{j}+t;x_{j},u_{x_{j}})\rvert.$$
		Then, if both $t_{j}+s,t_{j}+t<0$, by using the fact that $x_{j}\in \text{Lip}_{\bar{M}}([-\tau,0];\mathbb{R}^{n})$, we get

		$$\lvert x_{j+1}(s)-x_{j+1}(t)\rvert=\lvert x_{j}(t_{j}+s)-x_{j}(t_{j}+t)\rvert\leq \bar M\lvert s-t\rvert.$$
		On the other hand, if both $t_{0}+s,t_{0}+t>0$, assuming for simplicty that $t\leq s$ (the case $t>s$ is analogous) we can write

		$$\begin{array}{l}
		\vspace{0.3cm}\displaystyle{\lvert x_{j+1}(s)-x_{j+1}(t)\rvert}\\
		\vspace{0.3cm}\displaystyle{\hspace{1 cm}=\left\lvert y(0;x_{j},u_{x_{j}})+\int_{0}^{t_{j}+s}y'(r;x_{j},u_{x_{j}})dr-y(0;x_{j},u_{x_{j}})-\int_{0}^{t_{j}+t}y'(r;x_{j},u_{x_{j}})dr\right\rvert}\\
		\vspace{0.3cm}\displaystyle{\hspace{1 cm}\leq\int_{t_{j}+t}^{t_{j}+s}\lvert y'(r,x_{j},u_{x_{j}})\rvert dr=\int_{t_{j}+t}^{t_{j}+s}\lvert f(y(r-\tau;x_{j},u_{x_{j}}),u_{x_{j}})\rvert dr}\\
		\displaystyle{\hspace{1 cm}\leq M\lvert s-t\rvert\leq \bar M\lvert s-t\rvert.}
		\end{array}$$
		Finally, if $t_{j}+s>0$ and $t_{j}+t<0$ (or, analogously, if $t_{j}+s<0$ and $t_{j}+t>0$), it turns out that

		$$\begin{array}{l}
			\vspace{0.3cm}\displaystyle{\lvert x_{j+1}(s)-x_{j+1}(t)\rvert=\left\lvert y(0;x_{j},u_{x_{j}})+\int_{0}^{t_{j}+s}y'(r;x_{j},u_{x_{j}})dr- x_{j}(t_{j}+t)\right\rvert}\\
			\vspace{0.3cm}\displaystyle{\hspace{2.5 cm}\leq \lvert x_{j}(0)-x_{j}(t_{j}+t)\rvert+\int_{0}^{t_{j}+s}\lvert f(y(r-\tau;x_{j},u_{x_{j}}), u_{x_{j}})\rvert dr}\\
			\vspace{0.3cm}\displaystyle{\hspace{2.5 cm}\leq \bar M\lvert t_{j}+t\rvert+ M (t_{j}+s)}\\
			\vspace{0.3cm}\displaystyle{\hspace{2.5 cm}\leq \bar M\lvert t_{j}+t\rvert+ \bar M (t_{j}+s)}\\
			\vspace{0.3cm}\displaystyle{\hspace{2.5 cm}=\bar M(t_{j}+s-t_{j}-t)}\\
			\displaystyle{\hspace{2.5 cm}=\bar M\lvert s-t\rvert.}
		\end{array}$$
		Therefore $x_{j+1}\in \text{Lip}_{\bar{M}}([-\tau,0];\mathbb{R}^{n})$ as desired.
		Next, we claim that \begin{equation}\label{dist}
			d_{\mathcal{K}}(x_{j+1}(0))\leq k d_{\mathcal{K}}(x_{j}(0)),\quad \forall j\geq 0,
		\end{equation}
		where \begin{equation}\label{k}
			k=\sqrt{1-\mu\frac{\mu-2\bar ML\tau}{4M^{2}}}\in [0,1).
		\end{equation}
		To this aim, for all $j\geq 0$, we define \begin{equation}\label{yj}
			y_{j}(t)=y(t;x_{j},u_{x_{j}}),\quad \forall t\in [0,t_{j}]. 
		\end{equation}
		Given $j\geq 0$, for every $t\in [0,t_{j}]$, we have that

		$$\begin{array}{l}
			\vspace{0.3cm}\displaystyle{\frac{1}{2}\frac{d}{dt}\lvert y_{j}(t)-\pi(x_{j}(0))\rvert^{2}=\langle f(y_{j}(t-\tau),u_{x_{j}}),y_{j}(t)-\pi(x_{j}(0))\rangle}\\
			\vspace{0.3cm}\displaystyle{\hspace{2.5cm}=\langle f(x_{j}(0),u_{x_{j}}),x_{j}(0)-\pi(x_{j}(0))\rangle}\\
			\vspace{0.3cm}\displaystyle{\hspace{2.8cm}+\langle f(y_{j}(t-\tau),u_{x_{j}})-f(x_{j}(0),u_{x_{j}}), x_{j}(0)-\pi(x_{j}(0))\rangle}\\
			\displaystyle{\hspace{2.8 cm}+\langle f(y_{j}(t-\tau),u_{x_{j}}),y_{j}(t)-x_{j}(0)\rangle.}
		\end{array}$$
		Therefore, using $(H_{1})$, $(H_{3})$ and the Cauchy-Schwarz inequality, we get 

$$\begin{array}{l}
			\vspace{0.3cm}\displaystyle{\frac{1}{2}\frac{d}{dt}\lvert y_{j}(t)-\pi(x_{j}(0))\rvert^{2}\leq -\mu d_{\mathcal{K}}(x_{j}(0))+Ld_{\mathcal{K}}(x_{j}(0))\lvert y_{j}(t-\tau)-x_{j}(0)\rvert}\\
			\displaystyle{\hspace{4.6cm}+M\lvert y_{j}(t)-x_{j}(0)\rvert.}
		\end{array}$$
		We first estimate

		$$\begin{array}{l}
			\vspace{0.3cm}\displaystyle{\lvert y_{j}(t)-x_{j}(0)\rvert=\lvert y(t;x_{j},u_{x_{j}})-x_{j}(0)\rvert}\\
			\vspace{0.3cm}\displaystyle{\hspace{2.6cm}=\left\lvert x_{j}(0)+\int_{0}^{t}y'(s;x_{j},u_{x_{j}})ds-x_{j}(0)\right\rvert}\\
			\vspace{0.3cm}\displaystyle{\hspace{2.6cm}\leq \int_{0}^{t}\lvert f(y(s-\tau);x_{j},u_{x_{j}})\rvert ds}\\
			\displaystyle{\hspace{2.6cm}\leq Mt.}
		\end{array}$$
		Thus, \begin{equation}\label{1}
			\frac{1}{2}\frac{d}{dt}\lvert y_{j}(t)-\pi(x_{j}(0))\rvert^{2}\leq -\mu d_{\mathcal{K}}(x_{j}(0))+Ld_{\mathcal{K}}(x_{j}(0))\lvert y_{j}(t-\tau)-x_{j}(0)\rvert+M^{2}t.
		\end{equation}
		Moreover, 

		$$\lvert y_{j}(t-\tau)-x_{j}(0)\rvert=\lvert y(t-\tau;x_{j},u_{x_{j}})-x_{j}(0)\rvert.$$
		We can distinguish two different situations. Assume first that $t-\tau\leq0$. Then, since $x_{j}\in \text{Lip}_{\bar{M}}([-\tau,0];\mathbb{R}^{n})$, we have
	
	$$\lvert y_{j}(t-\tau)-x_{j}(0)\rvert=\lvert x_{j}(t-\tau)-x_{j}(0)\rvert\leq \bar M\lvert t-\tau\rvert=\bar M(\tau-t)\leq \bar M\tau.$$
		Therefore, \eqref{1} becomes

		$$\begin{array}{l}
			\vspace{0.3cm}\displaystyle{\frac{1}{2}\frac{d}{dt}\lvert y_{j}(t)-\pi(x_{j}(0))\rvert^{2}\leq -\mu d_{\mathcal{K}}(x_{j}(0))+\bar ML\tau d_{\mathcal{K}}(x_{j}(0))+M^{2}t}\\
			\vspace{0.3cm}\displaystyle{\hspace{2cm}\leq -\mu d_{\mathcal{K}}(x_{j}(0))+\bar ML\tau d_{\mathcal{K}}(x_{j}(0))+M^{2}t_{j}}\\
			\vspace{0.3cm}\displaystyle{\hspace{2cm}=-\mu d_{\mathcal{K}}(x_{j}(0))+\bar ML\tau d_{\mathcal{K}}(x_{j}(0))+M^{2}\frac{\mu-2\bar ML\tau}{4M^{2}}d_{\mathcal{K}}(x_{j}(0))}\\
			\vspace{0.3cm}\displaystyle{\hspace{2cm}\leq-\mu d_{\mathcal{K}}(x_{j}(0))+\bar ML\tau d_{\mathcal{K}}(x_{j}(0))+M^{2}\frac{\mu-2\bar ML\tau}{2M^{2}}d_{\mathcal{K}}(x_{j}(0))}\\
			\vspace{0.3cm}\displaystyle{\hspace{2cm}=-\mu d_{\mathcal{K}}(x_{j}(0))+\bar ML\tau d_{\mathcal{K}}(x_{j}(0))+\frac{\mu}{2}d_{\mathcal{K}}(x_{j}(0))-\bar ML\tau d_{\mathcal{K}}(x_{j}(0))}\\
			\displaystyle{\hspace{2cm}=-\frac{\mu}{2}d_{\mathcal{K}}(x_{j}(0)).}
		\end{array}$$
	On the other hand, assume that $t-\tau>0$. Then, 
	
$$\begin{array}{l}
		\vspace{0.3cm}\displaystyle{\lvert y_{j}(t-\tau)-x_{j}(0)\rvert=\left\lvert x_{j}(0)+\int_{0}^{t}y'(s;x_{j},u_{x_{j}})ds-x_{j}(0)\right\rvert}\\
		\vspace{0.3cm}\displaystyle{\hspace{3.4cm}\leq \int_{0}^{t}\lvert f(y(s-\tau;x_{j},u_{x_{j}}),u_{x_{j}})\rvert ds}\\
		\displaystyle{\hspace{3.4cm}\leq M(t-\tau)\leq Mt.}
	\end{array}$$
Using this and the fact that $x_{j}(0)\in \mathcal{K}_{\delta}$ with $\delta\leq \frac{M}{L}$ in \eqref{1}, we get

$$\begin{array}{l}
	\vspace{0.3cm}\displaystyle{\frac{1}{2}\frac{d}{dt}\lvert y_{j}(t)-\pi(x_{j}(0))\rvert^{2}\leq -\mu d_{\mathcal{K}}(x_{j}(0))+LMt d_{\mathcal{K}}(x_{j}(0))+M^{2}t}\\
	\vspace{0.3cm}\displaystyle{\hspace{2cm}\leq-\mu d_{\mathcal{K}}(x_{j}(0))+M^{2}t +M^{2}t =-\mu d_{\mathcal{K}}(x_{j}(0))+2M^{2}t}\\
	\vspace{0.3cm}\displaystyle{\hspace{2cm}\leq -\mu d_{\mathcal{K}}(x_{j}(0))+2M^{2}t_{j}}\\
	\vspace{0.3cm}\displaystyle{\hspace{2cm}\leq -\mu d_{\mathcal{K}}(x_{j}(0))+2M^{2}\frac{\mu}{4M^{2}}d_{\mathcal{K}}(x_{j}(0))}\\
	\displaystyle{\hspace{2cm}=-\frac{\mu}{2}d_{\mathcal{K}}(x_{j}(0)).}
\end{array}$$
We have so proven that
		\begin{equation}\label{dis1}
				\frac{1}{2}\frac{d}{dt}\lvert y_{j}(t)-\pi(x_{j}(0))\rvert^{2}\leq -\frac{\mu}{2}d_{\mathcal{K}}(x_{j}(0)), \quad \forall t \in [0,t_{j}].
		\end{equation}
		Thus, integrating \eqref{dis1} over $(0,t)$, we have that
		\begin{equation}\label{dis4}
			\begin{split}
				\lvert y_{j}(t)-\pi(x_{j}(0))\rvert^{2}&\leq -\mu d_{\mathcal{K}}(x_{j}(0))t+ \lvert y_{j}(0)-\pi(x_{j}(0))\rvert^{2}\\&=-\mu d_{\mathcal{K}}(x_{j}(0))t+\lvert x_{j}(0)-\pi(x_{j}(0))\rvert^{2}\\&=-\mu d_{\mathcal{K}}(x_{j}(0))t+d^{2}_{\mathcal{K}}(x_{j}(0)),
			\end{split}
		\end{equation}
		for all $t\in [0,t_{j}]$. In particular, for $t=t_{j}$, \eqref{dis4} yields

		$$\lvert y_{j}(t_{j})-\pi(x_{j}(0))\rvert^{2}\leq -\mu d_{\mathcal{K}}(x_{j}(0))t_{j}+d^{2}_{\mathcal{K}}(x_{j}(0)).$$
		As a consequence, we have 

		$$\begin{array}{l}
			\vspace{0.3cm}\displaystyle{d^{2}_{\mathcal{K}}(x_{j+1}(0))=d^{2}_{\mathcal{K}}(y_{j}(t_{j}))}\\
			\vspace{0.3cm}\displaystyle{\hspace{2.2cm}\leq\lvert y_{j}(t_{j})-\pi(x_{j}(0))\rvert^{2}}\\
			\vspace{0.3cm}\displaystyle{\hspace{2.2cm}\leq -\mu d_{\mathcal{K}}(x_{j}(0))t_{j}+d^{2}_{\mathcal{K}}(x_{j}(0))}\\
			\vspace{0.3cm}\displaystyle{ \hspace{2.2cm}=-\mu d_{\mathcal{K}}(x_{j}(0))\frac{\mu-2\bar ML\tau}{4M^{2}}d_{\mathcal{K}}(x_{j}(0))+d^{2}_{\mathcal{K}}(x_{j}(0))}\\
			\displaystyle{\hspace{2.2cm}=\left(1-\mu\,\frac{\mu-2\bar ML\tau}{4M^{2}}\right)d_{\mathcal{K}}^{2}(x_{j}(0)),}
		\end{array}$$
		from which follows that

		$$d_{\mathcal{K}}(x_{j+1}(0))\leq\sqrt{1-\mu\,\frac{\mu-2\bar ML\tau}{4M^{2}}}d_{\mathcal{K}}(x_{j}(0))=kd_{\mathcal{K}}(x_{j}(0)),$$
		i.e. \eqref{dist} is fulfilled.
		Now, thanks to an induction argument, \eqref{dist} implies that
		\begin{equation}\label{dist0}
			d_{\mathcal{K}}(x_{j}(0))\leq k^{j}d_{\mathcal{K}}(x_{0}(0)),
		\end{equation}
		for all $j\geq0$. Hence, using \eqref{dist0}, it turns out that \begin{equation}\label{2}
			d_{\mathcal{K}}(x_{j}(0))\to 0 \quad \text{as}\,\, j\to\infty.
		\end{equation}
		Furthermore, from \eqref{tj}, \eqref{k} and \eqref{dist0}, we have that

		$$\begin{array}{l}
			\vspace{0.3cm}\displaystyle{\sum_{j=0}^{\infty}t_{j}=\frac{\mu-2\bar ML\tau}{4M^{2}}\sum_{j=0}^{\infty}d_{\mathcal{K}}(x_{j}(0))\leq \frac{\mu-2\bar ML\tau}{4M^{2}}d_{\mathcal{K}}(x_{0}(0))\sum_{j=0}^{\infty}k^{j}}\\
			\displaystyle{\hspace{3cm}=\frac{\mu-2\bar ML\tau}{4M^{2}}d_{\mathcal{K}}(x_{0}(0))\frac{1}{1-k}.}
		\end{array}$$
		So, setting

		$$C=\frac{\mu-2\bar ML\tau}{4M^{2}}\frac{1}{1-k},$$
		it holds 
		\begin{equation}\label{3}
			\sum_{j=0}^{\infty}t_{j}\leq Cd_{\mathcal{K}}(x_{0}(0)).
		\end{equation}

		Now, let us define the control $\bar{u}:[0,+\infty[\rightarrow U$ as follows

		$$\bar{u}(t)=\begin{cases}u_{x_{h}},\quad\quad\quad\hspace{0,4 cm}\text{if}\, \sum_{j=0}^{h-1} t_{j}\leq t< \sum_{j=0}^{h} t_{j}, \,\text{for\, some}\, h\geq0,\\\text{arbitrary},\,\,\,\quad \text{if} \,\,t\geq \sum_{j=0}^{\infty} t_{j}.
		\end{cases}$$
		We can note that, for every $j\geq 0$, 

$$y\left(\sum_{h=0}^{j}t_{j};x_{0},\bar{u}\right)=y\left(\sum_{h=0}^{j}t_{j};x_{0},u_{x_{j+1}}\right)=x_{j+1}(0).$$
		Therefore, using \eqref{2} and \eqref{3}, we can conclude that

		$$T(x_{0})\leq\theta(x_{0},\bar{u})=\sum_{h=0}^{\infty}t_{j}\leq Cd_{\mathcal{K}}(x_{0}(0)).$$\end{proof}

{Now, we are able to prove that the minimum time functional is locally Lipschitz continuous in suitable subsets of the reachable set. Note that, with respect to the undelayed case we do not have here the local Lipschitz continuity of the minimum time function on  the reachable set but only on subsets of Lipschitz continuous data. Same remark applies to the semiconcavity property (see Theorem \ref{semiconcavitythm}). }

\begin{thm}\label{lip}
{Assume that  $(H_{1})$, $(H_{2})$, $(H_{4})$ hold and let $\bar M>0$. Let us denote with $M^*:=\max\{M,\bar{M}\}$ and 
	assume  $\displaystyle{\tau<\frac{\mu}{2 M^*L}}$. Then, $\mathcal{R}_{ M^*}$ is an open subset of $\text{Lip}_{{M}^*}([-\tau,0];\mathbb{R}^{n})$ and the minimum time functional $T$ is locally Lipschitz continuous in $\mathcal{R}_{M^*}$.}
\end{thm}
\begin{proof} {It is not restrictive to assume that $\bar{M}\geq M$. Thus, $M^*=\bar{M}$ and the smallness condition on the time delay size reads as $\tau<\frac{\mu}{2\bar{M}L}$. Fix $x_{0}\in \mathcal{R}_{\bar M}$.} Let $x$ be any function in $\text{Lip}_{\bar{M}}([-\tau,0];\mathbb{R}^{n})$ such that \begin{equation}\label{norm}
		\lVert x-x_{0}\rVert_{\infty}<\delta \frac{e^{-L(T(x_{0})+1)}}{(1+L\tau)},
	\end{equation}
where $\delta$ is the positive constant in Proposition \ref{stima_distanza}. We claim that
\begin{equation}\label{norma1}
	{T(x)\leq T(x_0)+ C(1+L\tau)e^{LT(x_{0})}\lVert x-x_{0}\rVert_{\infty}.}
\end{equation}
where $C>0$ is the positive constant in \eqref{estimate}.

Without loss of generality, we can assume that $T(x)>T(x_0)$, since otherwise \eqref{norma1} is trivially satisfied. Note that from the condition $T(x)>T(x_0)$, $T(x)$ is not necessary finite. Then, for $\epsilon>0$ small enough $T(x)\geq T(x_0)+\epsilon$. By definition of $T(x_{0})$, in correspondence of $\epsilon$, there exists a control $u_{\epsilon}$ such that \begin{equation}\label{inf1}
		\theta_{\epsilon}:=\theta(x_{0},u_{\epsilon})< T(x_{0})+\epsilon.
	\end{equation}
Therefore, $T(x)> \theta_{\epsilon}$. Moreover, using $(H_{1})$, for every $t\in [0,\theta_{\epsilon}]$, we get

$$\begin{array}{l}
	\vspace{0.3cm}\displaystyle{\lvert y(t;x,u_{\epsilon})-y(t;x_{0},u_{\epsilon})\rvert=\left\lvert y(0;x,u_{\epsilon})+\int_{0}^{t}y'(s;x,u_{\epsilon})ds-y(0;x_{0},u_{\epsilon})-\int_{0}^{t}y'(s;x_{0},u_{\epsilon})ds\right\rvert}\\
	\vspace{0.3cm}\displaystyle{\hspace{4.3cm}\leq\lvert x(0)-x_{0}(0)\rvert+\int_{0}^{t}\lvert y'(s;x,u_{\epsilon})-y'(s;x_{0},u_{\epsilon})\rvert ds}\\
	\vspace{0.3cm}\displaystyle{\hspace{4.3cm}\leq \lVert x-x_{0}\rVert_{\infty}+\int_{0}^{t}\lvert f(y(s-\tau;x,u_{\epsilon}),u_{\epsilon})-f(y(s-\tau;x_{0},u_{\epsilon}),u_{\epsilon})\rvert ds}\\
	\displaystyle{\hspace{4.3cm}\leq \lVert x-x_{0}\rVert_{\infty}+L\int_{0}^{t}\lvert y(s-\tau;x,u_{\epsilon})-y(s-\tau;x_{0},u_{\epsilon})\rvert ds.}
\end{array}$$
Now, we distinguish two situations. Assume first that $t\leq \tau$. Then, for all $s\in [0,t]$, it holds that $s-\tau\in [-\tau,0]$ and

	$$\lvert y(s-\tau;x,u_{\epsilon})-y(s-\tau;x_{0},u_{\epsilon})\rvert=\lvert x(s-\tau)-x_{0}(s-\tau)\rvert\leq \lVert x-x_{0}\rVert_{\infty}.$$	
	Thus, 

$$\lvert y(t;x,u_{\epsilon})-y(t;x_{0},u_{\epsilon})\rvert\leq \lVert x-x_{0}\rVert_{\infty}+L\lVert x-x_{0}\rVert_{\infty}t=(1+L\tau)\lVert x-x_{0}\rVert_{\infty}.$$
	On the contrary, assume that $t>\tau$. Then, using a change of variable, we can write

	$$\begin{array}{l}
		\vspace{0.3cm}\displaystyle{\int_{0}^{t}\lvert y(s-\tau;x,u_{\epsilon})-y(s-\tau;x_{0},u_{\epsilon})\rvert ds=\int_{-\tau}^{t-\tau}\lvert y(s;x,u_{\epsilon})-y(s;x_{0},u_{\epsilon})\rvert ds}\\
		\vspace{0.3cm}\displaystyle{\hspace{1.4cm}=\int_{-\tau}^{0}\lvert y(s;x,u_{\epsilon})-y(s;x_{0},u_{\epsilon})\rvert ds+\int_{0}^{t-\tau}\lvert y(s;x,u_{\epsilon})-y(s;x_{0},u_{\epsilon})\rvert  ds}\\
		\vspace{0.3cm}\displaystyle{\hspace{1.4cm}=\int_{-\tau}^{0}\lvert x(s)-x_{0}(s)\rvert ds+\int_{0}^{t-\tau}\lvert y(s;x,u_{\epsilon})-y(s;x_{0},u_{\epsilon})\rvert  ds}\\
		\displaystyle{\hspace{1.4cm}\leq \lVert x-x_{0}\rVert_{\infty}\tau+\int_{0}^{t}\lvert y(s;x,u_{\epsilon})-y(s;x_{0},u_{\epsilon})\rvert  ds.}
	\end{array}$$
	As a consequence, it comes that, for all $t\in [0, \theta_\epsilon],$
	
$$\begin{array}{l}
		\vspace{0.3cm}\displaystyle{\lvert y(t;x,u_{\epsilon})-y(t;x_{0},u_{\epsilon})\rvert\leq \lVert x-x_{0}\rVert_{\infty}}\\
		\vspace{0.3cm}\displaystyle{\hspace{5.3cm}+L\tau\lVert x-x_{0}\rVert_{\infty}+L\int_{0}^{t}\lvert y(s;x,u_{\epsilon})-y(s;x_{0},u_{\epsilon})\rvert  ds}\\
		\displaystyle{\hspace{4.8cm}\leq(1+L\tau)\lVert x-x_{0}\rVert_{\infty}+L\int_{0}^{t}\lvert y(s;x,u_{\epsilon})-y(s;x_{0},u_{\epsilon})\rvert  ds.}
	\end{array}$$
	Therefore, the Gronwall's inequality yields
	 \begin{equation}\label{estgron}
		\lvert y(t;x,u_{\epsilon})-y(t;x_{0},u_{\epsilon})\rvert\leq (1+L\tau)e^{Lt}\lVert x-x_{0}\rVert_{\infty},\quad \forall t\in [0,\theta_{\epsilon}].
	\end{equation}
	In particular, for $t=\theta_{\epsilon}$, we have

	$$\lvert y(\theta_{\epsilon};x,u_{\epsilon})-y(\theta_{\epsilon};x_{0},u_{\epsilon})\rvert\leq (1+L\tau)e^{L\theta_{\epsilon}}\lVert x-x_{0}\rVert_{\infty}.$$
	Thus, using \eqref{inf1}, we get \begin{equation}\label{lip1}
		\lvert y(\theta_{\epsilon};x,u_{\epsilon})-y(\theta_{\epsilon};x_{0},u_{\epsilon})\rvert\leq (1+L\tau)e^{L(T(x_{0})+\epsilon)}\lVert x-x_{0}\rVert_{\infty}.
	\end{equation} 
Let us note that $y_{\theta_{\epsilon}}(\cdot;x,u_{\epsilon})\in \text{Lip}_{\bar{M}}([-\tau,0];\mathbb{R}^{n})$. Indeed, let $s,t\in [-\tau,0]$. We can assume without loss of generality that $s\geq t$. Then, if both $\theta_{\epsilon}+s$, $\theta_{\epsilon}+t\leq0$, using the fact that $x\in \text{Lip}_{\bar{M}}([-\tau,0];\mathbb{R}^{n})$, we have that 

	$$\begin{array}{l}
		\vspace{0.3cm}\displaystyle{\lvert y_{\theta_{\epsilon}}(s;x,u_{\epsilon})-y_{\theta_{\epsilon}}(t;x,u_{\epsilon})\rvert=\lvert y(\theta_{\epsilon
			}+s;x,u_{\epsilon})-y(\theta_{\epsilon}+t;x,u_{\epsilon})\rvert}\\
		\displaystyle{\hspace{2cm}=\lvert x(\theta_{\epsilon}+s)-x(\theta_{\epsilon}+t)\rvert\leq \bar M\lvert s-t\rvert.}
	\end{array}$$
	On the other hand, if both $\theta_{\epsilon}+s,\theta_{\epsilon}+t>0$, from $(H_{1})$ we can write 

	$$\begin{array}{l}
		\vspace{0.3cm}\displaystyle{\lvert y_{\theta_{\epsilon}}(s;x,u_{\epsilon})-y_{\theta_{\epsilon}}(t;x,u_{\epsilon})\rvert=\lvert y(\theta_{\epsilon
			}+s;x,u_{\epsilon})-y(\theta_{\epsilon}+t;x,u_{\epsilon})\rvert}\\
		\vspace{0.3cm}\displaystyle{\hspace{4.8cm}=\left\lvert x(0)+\int_{0}^{\theta_{\epsilon}+s}y'(r;x,u_{\epsilon})dr-x(0)-\int_{0}^{\theta_{\epsilon}+t}y'(r;x,u_{\epsilon})dr\right\rvert}\\
		\vspace{0.3cm}\displaystyle{\hspace{4.8cm}\leq \int_{\theta_{\epsilon}+t}^{\theta_{\epsilon}+s}\lvert y'(r;x,u_{\epsilon})\rvert dr}\\
		\vspace{0.3cm}\displaystyle{\hspace{4.8cm}=\int_{\theta_{\epsilon}+t}^{\theta_{\epsilon}+s}\lvert f(y(r-\tau;x,u_{\epsilon}),u_{\epsilon})\rvert dr}\\
		\displaystyle{\hspace{4.8cm}\leq M (s-t)\leq \bar M(s-t)=\bar M\lvert s-t\rvert.}
	\end{array}$$
	Finally, if $\theta_{\epsilon}+s>0$ and $\theta_{\epsilon}+t\leq0$, using $(H_{1})$ and the fact that $x\in \text{Lip}_{\bar{M}}([-\tau,0];\mathbb{R}^{n})$, we have

	$$\begin{array}{l}
		\vspace{0.3cm}\displaystyle{\lvert y_{\theta_{\epsilon}}(s;x,u_{\epsilon})-y_{\theta_{\epsilon}}(t;x,u_{\epsilon})\rvert=\lvert y(\theta_{\epsilon
			}+s;x,u_{\epsilon})-y(\theta_{\epsilon}+t;x,u_{\epsilon})\rvert}\\
		\vspace{0.3cm}\displaystyle{\hspace{4.8cm}=\left\lvert x(0)+\int_{0}^{\theta_{\epsilon}+s}y'(r;x,u_{\epsilon})dr-x(\theta_{\epsilon}+t)\right\rvert}\\
		\vspace{0.3cm}\displaystyle{\hspace{4.8cm}\leq \lvert x(0)- x(\theta_{\epsilon}+t)\rvert+\int_{0}^{\theta_{\epsilon}+s}\lvert f(y(r-\tau;x,u_{\epsilon}),u_{\epsilon})\rvert dr}\\
		\vspace{0.3cm}\displaystyle{\hspace{4.8cm}\leq \bar M(-\theta_{\epsilon}-t)+M(\theta_{\epsilon}+s)}\\
		\displaystyle{\hspace{4.8cm}\leq \bar M(-\theta_{\epsilon}-t+\theta_{\epsilon}+s)=\bar M\lvert s-t\rvert.}
	\end{array}$$
	Thus, $y_{\theta_{\epsilon}}(\cdot;x,u_{\epsilon})\in \text{Lip}_{\bar{M}}([-\tau,0];\mathbb{R}^{n})$ as claimed before. In addition,  since $\epsilon$ is small, inequalities \eqref{norm} and \eqref{lip1} imply that

	$$\begin{array}{l}
		\vspace{0.3cm}\displaystyle{d_{\mathcal{K}}(y_{\theta_{\epsilon}}(0;x,u_{\epsilon}))=d_{\mathcal{K}}(y(\theta_{\epsilon};x,u_{\epsilon}))\leq \lvert y(\theta_{\epsilon};x,u_{\epsilon})-y(\theta_{\epsilon};x_{0},u_{\epsilon})\rvert}\\
		\vspace{0.3cm}\displaystyle{\hspace{6.1cm}\leq (1+L\tau)e^{L(T(x_{0})+\epsilon)}\lVert x-x_{0}\rVert_{\infty}}\\
		\vspace{0.3cm}\displaystyle{\hspace{6.1cm}\leq (1+L\tau)e^{L(T(x_{0})+1)}\lVert x-x_{0}\rVert_{\infty}}\\
		\displaystyle{\hspace{6.1cm}<(1+L\tau)e^{L(T(x_{0})+1)}\delta \frac{e^{-L(T(x_{0})+1)}}{1+L\tau}=\delta,}
	\end{array}$$
i.e. $y_{\theta_{\epsilon}}(0;x,u_{\epsilon})\in \mathcal{K}_{\delta}$. As a consequence, being $y_{\theta_{\epsilon}}(\cdot;x,u_{\epsilon})\in \text{Lip}_{\bar{M}}([-\tau,0];\mathbb{R}^{n})$ with $y_{\theta_{\epsilon}}(0;x,u_{\epsilon})\in \mathcal{K}_{\delta}$, estimate \eqref{estimate} yields
	\begin{equation}\label{5}
		T(y_{\theta_\epsilon}(\cdot;x,u_{\epsilon}))\leq Cd_{\mathcal{K}}(y_{\theta_\epsilon}(0;x,u_{\epsilon}))=Cd_{\mathcal{K}}(y(\theta_{\epsilon};x,u_{\epsilon})).
	\end{equation}
{In particular,
$$T(y_{\theta_\epsilon}(\cdot;x,u_{\epsilon}))<C\delta,$$and both $x,y_{\theta_\epsilon}(\cdot;x,u_{\epsilon})\in \mathcal{R}$. Now, we apply the dynamic programming principle (Proposition \ref{din}) to obtain}
\begin{equation}\label{lip2}
	T(x)\leq \theta_{\epsilon}+T(y_{\theta_\epsilon}(\cdot;x,u_{\epsilon})).
\end{equation}
Therefore, combining  \eqref{inf1}, \eqref{lip1}, \eqref{5} and \eqref{lip2}, we can conclude that 

$$\begin{array}{l}
	\vspace{0.3cm}\displaystyle{T(x)\leq \theta_{\epsilon} +Cd_{\mathcal{K}}(y(\theta_{\epsilon};x,u_{\epsilon}))}\\
	\vspace{0.3cm}\displaystyle{\hspace{1cm}\leq T(x_{0})+\epsilon+Cd_{\mathcal{K}}(y(\theta_{\epsilon};x,u_{\epsilon}))}\\
	\vspace{0.3cm}\displaystyle{\hspace{1cm}\leq T(x_{0})+\epsilon+C\lvert y(\theta_{\epsilon};x,u_{\epsilon})-y(\theta_{\epsilon};x_{0},u_{\epsilon})\rvert}\\
	\displaystyle{\hspace{1cm}\leq T(x_{0})+\epsilon+C(1+L\tau)e^{L(T(x_{0})+\epsilon)}\lVert x-x_{0}\rVert_{\infty}.}
\end{array}$$
Letting $\epsilon\to 0$, we finally get

 $$
 	T(x)-T(x_{0})\leq C(1+L\tau)e^{LT(x_{0})}\lVert x-x_{0}\rVert_{\infty},
$$
which proves \eqref{norma1}. {As a consequence, we  deduce that $\mathcal{R}_{\bar M}$ is an open subset of $\text{Lip}_{\bar{M}}([-\tau,0];\mathbb{R}^{n}),$ due to the fact that, from \eqref{norma1}, for all $x_0\in \mathcal{R}_{\bar{M}}$}
$${B_{C([-\tau,0];\mathbb{R}^{n})}\left(x_{0},\delta \frac{e^{-L(T(x_{0})+1)}}{1+L\tau}\right)\cap\text{Lip}_{\bar{M}}([-\tau,0];\mathbb{R}^{n})\subset \mathcal{R}_{\bar{M}}, }$$
{where $B_{C([-\tau,0];\mathbb{R}^{n})}\left(x_{0},\delta \frac{e^{-L(T(x_{0})+1)}}{1+L\tau}\right)$ is the open ball in the Banach space $C([-\tau,0];
\mathbb{R}^{n})$ centered at $x_{0}$ with radius $\delta \frac{e^{-L(T(x_{0})+1)}}{1+L\tau}$.}
\\Now, we prove the local Lipschitz continuity of $T$. Let $x_0\in \mathcal{R}_{\bar{M}}$. For every $x\in \text{Lip}_{\bar{M}}([-\tau,0];\mathbb{R}^{n})$ satisfying \eqref{norm}, from \eqref{norma1} it holds

$$\begin{array}{l}
	\vspace{0.3cm}\displaystyle{ T(x)\leq  T(x_{0})+C(1+L\tau)e^{LT(x_{0})}\lVert x-x_{0}\rVert_{\infty}}\\
	\vspace{0.3cm}\displaystyle{\hspace{1cm}<  T(x_{0})+C(1+L\tau)e^{LT(x_{0})} \delta \frac{e^{-L(T(x_{0})+1)}}{1+L\tau}}\\
	\displaystyle{\hspace{1cm}=T(x_{0})+C\delta e^{-L}.}
\end{array}$$
So, setting $\bar{c}:= T(x_{0})+C\delta e^{-L},$ we have that 
\begin{equation}\label{tmboun}
	{T(x)\leq \bar{c},\quad \forall x\in B_{C([-\tau,0];\mathbb{R}^{n})}\Big (x_{0},\delta \frac{e^{-L(T(x_{0})+1)}}{1+L\tau}\Big )\cap \text{Lip}_{\bar{M}}([-\tau,0];\mathbb{R}^{n}).}
\end{equation}
We claim that 
\begin{equation}\label{stimalip}
	{\lvert T(\bar x)-T(\tilde{x})\rvert\leq c\lVert \bar x-\tilde{x}\rVert_{\infty},\quad \forall \bar{x},\tilde{x}\in  B_{C([-\tau,0];\mathbb{R}^{n})}\Big (x_{0},\delta \frac{e^{-L(\bar{c}+1)}}{2(1+L\tau)}\Big )\cap\mathcal{R}_{\bar M}.}
\end{equation}
for some positive constant $c$.
\\To this aim, let $\bar{x},\tilde{x}\in  B_{C([-\tau,0];\mathbb{R}^{n})}\left(x_{0},\delta \frac{e^{-L(\bar{c}+1)}}{2(1+L\tau)}\right)\cap\mathcal{R}_{\bar M}$. We first assume that $T(\bar{x})>T(\tilde{x})$. Then, for $\epsilon>0$ small, $T(\bar{x})\geq T(\tilde{x})+\epsilon$. Also, in correspondence of $\epsilon$, by definition of the value function, there exists a control $u_{\epsilon}$ such that 
\begin{equation}\label{teps}
	\theta_{\epsilon}:=\theta(\tilde{x},u_{\epsilon})<T(\tilde{x})+\epsilon.
\end{equation}
With the same arguments employed for the proof of \eqref{estgron}, it holds
\begin{equation}\label{dist1}
	\lvert y(t;\bar{x},u_{\epsilon})-y(t;\tilde{x},u_{\epsilon})\rvert\leq (1+L\tau)e^{Lt}\lVert \bar{x}-\tilde{x}\rVert_{\infty},\quad \forall t \in [0,\theta_{\epsilon}].
\end{equation}
In addition, the dynamic programming principle (Proposition \ref{din}) yields 
\begin{equation}\label{dist2}
	T(\bar{x})\leq \theta_{\epsilon}+T(y_{\theta_{\epsilon}}(\cdot ;\bar{x},u_{\epsilon})).
\end{equation}
Let us note that, from \eqref{tmboun} and \eqref{dist1} it comes that 

$$\begin{array}{l}
	\vspace{0.3cm}\displaystyle{d_{\mathcal{K}}(y_{\theta_{\epsilon}}(0 ;\bar{x},u_{\epsilon}))=d_{\mathcal{K}}(y(\theta_{\epsilon} ;\bar{x},u_{\epsilon}))\leq \lvert y(\theta_{\epsilon} ;\bar{x},u_{\epsilon})-y(\theta_{\epsilon} ;\tilde{x},u_{\epsilon})\rvert}\\
	\vspace{0.3cm}\displaystyle{\hspace{5.7cm}\leq (1+L\tau)e^{L\theta_{\epsilon}}\lVert \bar{x}-\tilde{x}\rVert_{\infty}}\\
	\vspace{0.3cm}\displaystyle{\hspace{5.7cm}\leq (1+L\tau)e^{L(T(\tilde{x})+\epsilon)}\lVert \bar{x}-\tilde{x}\rVert_{\infty}}\\
	\displaystyle{\hspace{5.7cm}\leq (1+L\tau)e^{L(\bar{c}+1)}\lVert \bar{x}-\tilde{x}\rVert_{\infty}.}
\end{array}$$
Thus, since 

$$\lVert \bar{x}-\tilde{x}\rVert_{\infty}\leq \lVert \bar{x}-x_{0}\rVert_{\infty}+\lVert x_{0}-\tilde{x}\rVert_{\infty}<\delta\frac{e^{L(\bar{c}+1)}}{2(1+L\tau)}+ \delta\frac{e^{L(\bar{c}+1)}}{2(1+L\tau)}=\delta\frac{e^{-L(\bar{c}+1)}}{1+L\tau},$$
we get

$$d_{\mathcal{K}}(y_{\theta_{\epsilon}}(0 ;\bar{x},u_{\epsilon}))=d_{\mathcal{K}}(y(\theta_{\epsilon} ;\bar{x},u_{\epsilon}))< (1+L\tau)e^{L(\bar{c}+1)}\delta\frac{e^{-L(\bar{c}+1)}}{1+L\tau}=\delta,$$
i.e. $y_{\theta_{\epsilon}}(0,\bar{x},u_{\epsilon})\in \mathcal{K}_{\delta}$. Therefore, being $y_{\theta_{\epsilon}}(\cdot,\bar{x},u_{\epsilon})\in \mathcal{R}_{\bar M}$ from the fact that $\bar{x}\in\mathcal{R}_{\bar M}$, \eqref{estimate} yields

$$T(y_{\theta_{\epsilon}}(\cdot,\bar{x},u_{\epsilon}))\leq Cd_{\mathcal{K}}(y_{\theta_{\epsilon}}(0 ;\bar{x},u_{\epsilon})).$$
Finally, combining this last fact with \eqref{tmboun}, \eqref{teps}, \eqref{dist1},  we can write

$$\begin{array}{l}
	\vspace{0.3cm}\displaystyle{T(\bar{x})\leq T(\tilde{x})+\epsilon+T(y_{\theta_{\epsilon}}(\cdot,\bar{x},u_{\epsilon}))}\\
	\vspace{0.3cm}\displaystyle{\hspace{1cm}\leq T(\tilde{x})+\epsilon+Cd_{\mathcal{K}}(y_{\theta_{\epsilon}}(0 ;\bar{x},u_{\epsilon}))}\\
	\vspace{0.3cm}\displaystyle{\hspace{1cm}\leq T(\tilde{x})+\epsilon+C\lvert y(\theta_{\epsilon};\bar{x},u_{\epsilon})-y(\theta_{\epsilon};\tilde{x},u_{\epsilon})\rvert}\\
	\vspace{0.3cm}\displaystyle{\hspace{1cm}\leq T(\tilde{x})+\epsilon+C(1+L\tau)e^{L\theta_{\epsilon}}\lVert \bar{x}-\tilde{x}\rVert_{\infty}}\\
	\vspace{0.3cm}\displaystyle{\hspace{1cm}\leq T(\tilde{x})+\epsilon+C(1+L\tau)e^{L(T(\tilde{x})+\epsilon)}\lVert \bar{x}-\tilde{x}\rVert_{\infty}}\\
	\displaystyle{\hspace{1cm}\leq T(\tilde{x})+\epsilon+C(1+L\tau)e^{L(\bar{c}+\epsilon)}\lVert \bar{x}-\tilde{x}\rVert_{\infty}.}
\end{array}$$
Letting $\epsilon\to 0$, we get

$$T(\bar{x})-T(\tilde{x})\leq C(1+L\tau)e^{L\bar{c}}\lVert \bar{x}-\tilde{x}\rVert_{\infty}.$$
So, setting $c:=C(1+L\tau)e^{L\bar{c}}$, we can conclude that 

$$T(\bar{x})-T(\tilde{x})\leq c\lVert \bar{x}-\tilde{x}\rVert_{\infty}.$$
Now, if $T(\bar{x})<T(\tilde{x})$, exchanging the roles of $\bar{x},\tilde{x}$, in the above arguments, it also holds that 

$$T(\tilde{x})-T(\bar{x})\leq c\lVert \bar{x}-\tilde{x}\rVert_{\infty}.$$
Thus, \eqref{stimalip} is fulfilled and $T$ is Lipschitz continuous in a neighborhood of $x_{0}$. From the arbitrariness of $x_{0}$, the minimum time functional is locally Lipschitz continuous in $\mathcal{R}_{\bar M}$.
\end{proof}

\begin{oss}
{To the aim of our analysis, in the sequel, we will need the Lipschitz continuity of $T(\cdot)$ in ${\mathcal R}_{3 M}.$ As a consequence, by virtue of Theorem \ref{lip}, our semiconcavity result will be proved under the following condition on the time delay size:
$\displaystyle{\tau<\frac{\mu}{6ML}}.$}
\end{oss}

\section{Semiconcavity of the value function}
In this section, we present our main result, which guarantees that the minimum time functional $T$ is semiconcave in a suitable subset of the reachable set $\mathcal{R}$. The semiconcavity result we will prove is the following.
\begin{thm}\label{semiconcavitythm}
	Assume that $\tau<\frac{\mu}{6ML}$. Moreover, suppose that assumptions $(H_{1})$, $(H_{2})$, $(H_{3})$, $(H_{4})$ are satisfied and that the target $\mathcal{K}$ is such that \begin{equation}\label{neccond}
		d_{\mathcal{K}}\,\text{is semiconcave in }\, \mathbb{R}^{n}\setminus\mathring{\mathcal{K}},
	\end{equation}
	where $\mathring{\mathcal{K}}$ denotes the interior of $\mathcal{K}$. Then, the minimum time function $T$  is semiconcave in $\mathcal{R}_{M}\setminus\mathcal{R}_{M}^{\mathcal{K}}$, where ${\mathcal{R}_{M}^{\mathcal{K}}=\{x\in \mathcal{R}_{M}:x(0)\in \mathcal{K}\}}$. 
\end{thm}
For the proof of Theorem \ref{semiconcavitythm}, we need the following auxiliary lemma.
\begin{lem}\label{auxlem}
	Let the hypotheses of Theorem \ref{semiconcavitythm} hold. Then, there exist $\rho>0$, $k>0$ such that $$T(x+2h)-2T(x+h)\leq k\lVert h\rVert_{\infty}^{2},$$
	for every $x\in \mathcal{R}$ and $h\in C([-\tau,0];\mathbb{R}^{n})$ satisfying the following conditions: $x(0)\in \partial\mathcal{K}$; $\lVert h\rVert_{\infty}\leq\rho$; $(x+h)(0)\notin \mathcal{K}$; $x,x+h\in  \mathcal{R}_{M}$, $x+2h\in\mathcal{R}$.
\end{lem}
\begin{proof}
	First of all, let us note that, from \eqref{neccond} and from the compactness of $\mathcal{K}$, there exist two constants $r>0$, $c>0$ such that
	\begin{equation}\label{dk}
		d_{\mathcal{K}}(z_{1})+d_{\mathcal{K}}(z_{2})-2d_{\mathcal{K}}\left(\frac{z_{1}+z_{2}}{2}\right)\leq c\lvert z_{1}-z_{2}\rvert^{2},
	\end{equation}
	for every $z_{1},\,z_{2}$ satisfying $z_{1},\,z_{2},\,\frac{z_{1}+z_{2}}{2}\in \mathcal{K}_{r}\setminus\mathring{\mathcal{K}}.$
	Now, we pick $\rho >0$ such that \begin{equation}\label{rho}
		2(1+2MC)\rho<\min \{\delta,r\},
	\end{equation}
 where $C$ and $\delta$ are the constants  of Proposition \ref{stima_distanza} in the case $\bar M=M.$
	Let $x\in \mathcal{R}$ and $h\in C([-\tau,0];\mathbb{R}^{n})$ satisfying  the following properties: $x(0)\in \partial\mathcal{K}$, $\lVert h\rVert_{\infty}\leq\rho$, $(x+h)(0)\notin \mathcal{K}$, $x,\,x+h\in \mathcal{R}_{M}$, $x+2h\in \mathcal{R}$. 
\\Being $(x+h)(0)\notin\mathcal{K}$, we can thereby consider
	
$$T(x+h)=\inf_{u}\theta(x+h,u).$$
	By definition of $T$, { fixed} $\epsilon\in (0,C\lVert h\rVert_{\infty})$, there exists a control $u$ such that \begin{equation}\label{theta*}
		\theta^{*}:=\theta(x+h,u)<T(x+h)+\epsilon.
	\end{equation}
	Let us define

	$$\bar{u}(t)=u\left(\frac{t}{2}\right),\quad \tilde{y}(t)=y(t;x+h,u),\quad \bar{y}(t)=y(t;x+2h,\bar{u}).$$
	We claim that $\tilde{y}(t),\,\bar{y}(t)\in \mathcal{K}_{\delta}\cap\mathcal{K}_{r}$, for every $t\in [0,2\theta^{*}]$. Indeed, fix $t\in [0,2\theta^{*}]$. Then, since $x(0)\in \partial\mathcal{K}$, we can write

	$$\begin{array}{l}
		\vspace{0.3cm}\displaystyle{d_{\mathcal{K}}(\tilde{y}(t))\leq \lvert\tilde{y}(t)- x(0)\rvert=\lvert y(t;x+h,u)- x(0)\rvert}\\
		\vspace{0.3cm}\displaystyle{\hspace{1.6cm}=\left\lvert y(0;x+h,u)+\int_{0}^{t}\tilde{y}'(s)ds- x(0)\right\rvert}\\
		\vspace{0.3cm}\displaystyle{\hspace{1.6cm}=\left\lvert x(0)+h(0)+\int_{0}^{t}\tilde{y}'(s)ds- x(0)\right\rvert}\\
		\vspace{0.3cm}\displaystyle{\hspace{1.6cm}\leq \lvert h(0)\rvert+\int_{0}^{t}\lvert\tilde{y}'(s)\rvert ds }\\
		\vspace{0.3cm}\displaystyle{\hspace{1.6cm}\leq \lVert h\rVert_{\infty}+\int_{0}^{t}\lvert f(\tilde{y}(s-\tau),u)\rvert ds}\\
		\displaystyle{\hspace{1.6cm}\leq \lVert h\rVert_{\infty}+Mt\leq \lVert h\rVert_{\infty}+2M\theta^{*}.}
	\end{array}$$
	Now, using \eqref{theta*}, we get

	$$ 2M\theta^{*}< 2M(T(x+h)+\epsilon),$$
	from which, being $\epsilon\in (0,C\lVert h\rVert_{\infty})$, \begin{equation}\label{6}
		2M\theta^{*}<2M(T(x+h)+C\lVert h\rVert_{\infty}).
	\end{equation}
	In addition, we have that $(x+h)(0)\in \mathcal{K}_{\delta}\setminus\mathcal{K}$. Indeed, by using \eqref{rho} and the fact that $x(0)\in \partial\mathcal{K}$, it turns out that 

$$d_{\mathcal{K}}((x+h)(0))\leq\lvert x(0)+h(0)-x(0)\rvert\leq \lVert h\rVert_{\infty}\leq\rho <\delta.$$
	Thus, $x+h$, which is Lipschitz continuous of constant $M$, is such that $(x+h)(0)\in \mathcal{K}_{\delta}\backslash\mathcal{K}$. So, estimate \eqref{estimate} yields

	$$T(x+h)\leq Cd_{\mathcal{K}}((x+h)(0))\leq C\lVert h\rVert_{\infty}.$$
	Therefore, \eqref{6} becomes 

$$2M\theta^{*}<2M( C\lVert h\rVert_{\infty}+ C\lVert h\rVert_{\infty})=4MC\lVert h\rVert_{\infty},$$
and so
	\begin{equation}\label{7}
	\theta^{*}<2 C\lVert h\rVert_{\infty}.
\end{equation}
As a consequence, using \eqref{rho} and \eqref{7}, we can write 

	$$\begin{array}{l}
		\vspace{0.3cm}\displaystyle{d_{\mathcal{K}}(\tilde{y}(t))\leq\lVert h\rVert_{\infty}+4MC\lVert h\rVert_{\infty}}\\
		\vspace{0.3cm}\displaystyle{\hspace{1.6cm}\leq 2\lVert h\rVert_{\infty}+4MC\lVert h\rVert_{\infty}}\\
		\vspace{0.3cm}\displaystyle{\hspace{1.6cm}=2(1+2MC)\lVert h\rVert_{\infty}}\\
		\displaystyle{\hspace{1.6cm}\leq 2(1+2MC)\rho <\min\{\delta,r \},}
	\end{array}$$
	i.e. $\tilde{y}(t)\in \mathcal{K}_{\delta}\cap\mathcal{K}_{r}$.
	On the other hand, since $x(0)\in \partial\mathcal{K}$, we have that

	$$\begin{array}{l}
		\vspace{0.3cm}\displaystyle{d_{\mathcal{K}}(\bar{y}(t))\leq \lvert\bar{y}(t)- x(0)\rvert=\lvert y(t;x+2h,\bar{u})- x(0)\rvert}\\
		\vspace{0.3cm}\displaystyle{\hspace{1.6cm}=\left\lvert y(0;x+2h,\bar{u})+\int_{0}^{t}\bar{y}'(s)ds- x(0)\right\rvert}\\
		\vspace{0.3cm}\displaystyle{\hspace{1.6cm}\leq 2\lvert h(0)\rvert+\int_{0}^{t}\lvert f(\bar{y}(s-\tau),\bar{u})\rvert ds}\\
		\displaystyle{\hspace{1.6cm}\leq 2\lVert h\rVert_{\infty}+Mt\leq2\lVert h\rVert_{\infty}+2M\theta^{*}.}
	\end{array}$$
	Thus, using \eqref{7}, we get

	$$\begin{array}{l}
		\vspace{0.3cm}\displaystyle{d_{\mathcal{K}}(\bar{y}(t))\leq2\lVert h\rVert_{\infty}+4MC\lVert h\rVert_{\infty}=2(1+2MC)\lVert h\rVert_{\infty}}\\
		\displaystyle{\hspace{2cm}\leq2(1+2MC)\rho<\min\{\delta,r \},}
	\end{array}$$
	i.e. $\bar{y}(t)\in \mathcal{K}_{\delta}\cap\mathcal{K}_{r}$.
	Next, we claim that \begin{equation}\label{8}
		T(x+2h)\leq 2(T(x+h)+\epsilon)+k\lVert h\rVert_{\infty}^{2}
	\end{equation}
	for a suitable constant $k>0$, indipendent of $\epsilon,\,x,\,h$. Let us note that we can assume $2\theta^{*}<\theta(x+2h,\bar{u})$. Indeed, if this is not the case, it holds

	$$T(x+2h)\leq \theta(x+2h,\bar{u})\leq 2\theta^{*}<2(T(x+h)+\epsilon)$$
	and \eqref{8} is satisfied for any positive constant $k$. 
	\\Therefore, using the dynamic programming principle, we can write
	\begin{equation}\label{pp}
		T(x+2h)\leq 2\theta^{*}+T(\bar{y}_{2\theta^{*}}(\cdot)).
	\end{equation}
	Now, for every $t_{1}\in [0, \theta^{*}]$ and $t_{2}\in [0,2\theta^{*}]$, we compute
	$$\begin{array}{l}
		\vspace{0.3cm}\displaystyle{\lvert \tilde{y}(t_{1})-\bar{y}(t_{2})\rvert=\lvert y(t_{1};x+h,u)-y(t_{2};x+2h,\bar{u})\rvert}\\
		\vspace{0.3cm}\displaystyle{\hspace{1cm}=\left\lvert x(0)+h(0)+\int_{0}^{t_{1}}y'(s;x+h,u)ds-x(0)-2h(0)-\int_{0}^{t_{2}}y'(s;x+2h,\bar{u})ds\right\rvert}\\
		\vspace{0.3cm}\displaystyle{\hspace{1cm}\leq \lvert h(0)\rvert+\int_{0}^{t_{1}}\lvert f(\tilde{y}(s-\tau),u)\rvert ds+\int_{0}^{t_{2}}\lvert f(\bar{y}(s-\tau),\bar{u})\rvert ds}\\
		\displaystyle{\hspace{1cm}\leq \lVert h\rVert_{\infty}+Mt_{1}+Mt_{2}\leq \lVert h\rVert_{\infty}+M\theta^{*}+2M\theta^{*}=\lVert h\rVert_{\infty}+3M\theta^{*}.}
	\end{array}$$
	Therefore, since \eqref{7} implies that

 $$3M\theta^{*}<6MC\lVert h\rVert_{\infty},$$
	we can write
	\begin{equation}\label{9}
		\lvert \tilde{y}(t_{1})-\bar{y}(t_{2})\rvert\leq (1+6MC)\lVert h\rVert_{\infty}.
	\end{equation}
	Similarly, it holds

	$$\lvert \tilde{y}(\theta^{*})-x(0)\rvert\leq {\Vert h\Vert_\infty}+M\theta^{*},$$
	from which, by using again \eqref{7}, it comes that
	\begin{equation}\label{10}
		\lvert \tilde{y}(\theta^{*})-x(0)\rvert\leq (1+2MC)\lVert h\rVert_{\infty}.
	\end{equation}
	Also, for every $t\in [0,2\theta^{*}]$, we have

	$$\begin{array}{l}
		\vspace{0.3cm}\displaystyle{\lvert\bar{y}(t)-x(0)\rvert=\lvert y(t;x+2h,\bar{u})-x(0)\rvert}\\
		\vspace{0.3cm}\displaystyle{\hspace{2.3cm}=\left\lvert x(0)+2h(0)+\int_{0}^{t}\bar{y}'(s)ds-x(0)\right\rvert}\\
		\vspace{0.3cm}\displaystyle{\hspace{2.3cm}\leq 2\lVert h\rVert_{\infty}+\int_{0}^{t}\lvert f(\bar{y}(s-\tau),\bar{u})\rvert}\\
		\displaystyle{\hspace{2.3cm}\leq 2\lVert h\rVert_{\infty}+Mt\leq2\lVert h\rVert_{\infty}+2M \theta^{*}.}
	\end{array}$$
	Thus, from \eqref{7} it follows that
	\begin{equation}\label{11}
		\lvert\bar{y}(t)-x(0)\rvert\leq2\lVert h\rVert_{\infty}+4MC \lVert h\rVert_{\infty}=2(1+2MC)\lVert h\rVert_{\infty}.
	\end{equation}
	Now, let $\tilde{x}$, $\bar{x}$ be the functions in $C([-\tau,0];\mathbb{R}^{n})$ defined as follows 

$$\tilde{x}(\cdot)=\tilde{y}_{\theta^{*}}(\cdot),\quad \bar{x}(\cdot)=\bar{y}_{2\theta^{*}}(\cdot).$$
	Then,
	
$$\tilde{x}(0)=\tilde{y}_{\theta^{*}}(0)=y(\theta^{*};x+h,u)=y(\theta(x+h,u);x+h,u)\in \partial \mathcal{K}.$$
	Moreover, it holds
	
$$\begin{array}{l}
		\vspace{0.3cm}\displaystyle{\lvert \bar{x}(0)+x(0)-2\tilde{x}(0) \rvert=\lvert\bar{y}(2\theta^{*})+x(0)-2\tilde{y}(\theta^{*})\rvert}\\
		\vspace{0.3cm}\displaystyle{\hspace{1cm}=\lvert y(2\theta^{*};x+2h,\bar{u})+x(0)-2y(\theta^{*};x+h,u)\rvert}\\
		\vspace{0.3cm}\displaystyle{\hspace{1cm}=\left\lvert x(0)+2h(0)+\int_{0}^{2\theta^{*}}\bar{y}'(s)ds +x(0)-2x(0)-2h(0)-2\int_{0}^{\theta^{*}}\tilde{y}'(s)ds\right\rvert}\\
		\vspace{0.3cm}\displaystyle{\hspace{1cm}=\left\lvert2\int_{0}^{\theta^{*}}\bar{y}'(2s)ds-2\int_{0}^{\theta^{*}}\tilde{y}'(s)ds\right\rvert}\\
		\vspace{0.3cm}\displaystyle{\hspace{1cm}\leq2\int_{0}^{\theta^{*}}\lvert f(\bar{y}(2s-\tau),\bar{u}(2s))-f(\tilde{y}(s{-\tau}),u(s)) \rvert ds}\\
		\vspace{0.3cm}\displaystyle{\hspace{1cm}=2\int_{0}^{\theta^{*}}\lvert f(\bar{y}(2s-\tau),u(s))-f(\tilde{y}(s{-\tau}),u(s) \rvert ds }\\
		\vspace{0.3cm}\displaystyle{\hspace{1cm}\leq 2L\int_{0}^{\theta^{*}}\lvert\bar{y}(2s-\tau)-\tilde{y}(s-\tau)\rvert ds}\\
		\displaystyle{\hspace{1cm}=2L\int_{-\tau}^{\theta^{*}-\tau}\lvert\bar{y}(2s+\tau)-\tilde{y}(s)\rvert ds.}
	\end{array}$$
	Now, we distinguish two different situations. Assume first that $2\theta^{*}\leq \tau$. Then, also $\theta^{*}\leq \tau$ and we can write

	$$\begin{array}{l}
		\vspace{0.3cm}\displaystyle{\lvert \bar{x}(0)+x(0)-2\tilde{x}(0) \rvert\leq2L\int_{-\tau}^{\theta^{*}-\tau}\lvert x(2s+\tau)+2h(2s+\tau)-x(s)-h(s)\rvert ds}\\
		\displaystyle{\hspace{1cm}\leq 2L\int_{-\tau}^{\theta^{*}-\tau}\lvert x(2s+\tau)-x(s)\rvert ds+2L\int_{-\tau}^{\theta^{*}-\tau}\lvert 2h(2s+\tau)-h(s)\rvert ds.}
	\end{array}$$
	Thus, since $x\in \mathcal{R}_{M}$, it comes that

	$$\begin{array}{l}
		\vspace{0.3cm}\displaystyle{\lvert \bar{x}(0)+x(0)-2\tilde{x}(0) \rvert\leq 2LM\int_{-\tau}^{\theta^{*}-\tau}\lvert 2s+\tau-s\rvert ds+6L\lVert h\rVert_{\infty}\theta^{*}}\\
		\vspace{0.3cm}\displaystyle{\hspace{2cm}=2LM\int_{-\tau}^{\theta^{*}-\tau}(s+\tau) ds+6L\lVert h\rVert_{\infty}\theta^{*}}\\
		\displaystyle{\hspace{2cm}{ =LM(\theta^{*})^{2}}+6L\lVert h\rVert_{\infty}\theta^{*}.}
	\end{array}$$
	So, using \eqref{7}, we can conclude that
	\begin{equation}\label{12}
		\lvert \bar{x}(0)+x(0)-2\tilde{x}(0) \rvert\leq {4}MLC^{2}\lVert h\rVert_{\infty}^{2}+12LC\lVert h\rVert_{\infty}^{2}=4LC({MC}+3)\lVert h\rVert_{\infty}^{2}.
	\end{equation}
	On the contrary, assume that $2\theta^{*}>\tau$. In this case, we have to examine two further situations. Firstly, if $\theta^{*}\leq \tau$, using again the fact that $x\in \mathcal{R}_{M}$, we have that

	$$\begin{array}{l}
		\vspace{0.3cm}\displaystyle{\lvert \bar{x}(0)+x(0)-2\tilde{x}(0)\rvert\leq 2L\int_{-\tau}^{-\frac{\tau}{2}}\lvert \bar{y}(2s+\tau)-\tilde{y}(s)\rvert ds+2L\int_{-\frac{\tau}{2}}^{\theta^{*}-\tau}\lvert \bar{y}(2s+\tau)-\tilde{y}(s)\rvert ds}\\
		\vspace{0.3cm}\displaystyle{\hspace{1cm}=2L\int_{-\tau}^{-\frac{\tau}{2}}\lvert (x+2h)(2s+\tau)-(x+h)(s)\rvert ds+2L\int_{-\frac{\tau}{2}}^{\theta^{*}-\tau}\lvert \bar{y}(2s+\tau)-(x+h)(s)\rvert ds}\\
		\vspace{0.3cm}\displaystyle{\hspace{1cm}\leq 2L\int_{-\tau}^{-\frac{\tau}{2}}\lvert x(2s+\tau)-x(s)\rvert ds+2L\int_{-\tau}^{-\frac{\tau}{2}}\lvert 2h(2s+\tau)-h(s)\rvert ds}\\
		\vspace{0.3cm}\displaystyle{\hspace{1.3cm}+2L\int_{-\frac{\tau}{2}}^{\theta^{*}-\tau}\lvert \bar{y}(2s+\tau)-x(0)\rvert ds+2L\int_{-\frac{\tau}{2}}^{\theta^{*}-\tau}\lvert x(0)-x(s)-h(s)\rvert ds,}
\end{array}
$$
and so, 

 $$\begin{array}{l}
		\vspace{0.3cm}\displaystyle{\lvert \bar{x}(0)+x(0)-2\tilde{x}(0)\rvert}\\
		\vspace{0.3cm}\displaystyle{\hspace{1cm}\leq 2LM\int_{-\tau}^{-\frac{\tau}{2}}\lvert 2s+\tau-s\rvert ds+6L\lVert h\rVert_{\infty}\frac{\tau}{2}+2L\int_{-\frac{\tau}{2}}^{\theta^{*}-\tau}\lvert \bar{y}(2s+\tau)-x(0)\rvert ds}\\
		\vspace{0.3cm}\displaystyle{\hspace{1.3cm}+2L\int_{-\frac{\tau}{2}}^{\theta^{*}-\tau}\lvert x(0)-x(s)\rvert ds+2L\int_{-\frac{\tau}{2}}^{\theta^{*}-\tau}\lvert h(s)\rvert ds}\\
		\vspace{0.3cm}\displaystyle{\hspace{1cm}= 2LM\int_{-\tau}^{-\frac{\tau}{2}}(s+\tau)ds+6L\lVert h\rVert_{\infty}\frac{\tau}{2}+2L\int_{-\frac{\tau}{2}}^{\theta^{*}-\tau}\lvert \bar{y}(2s+\tau)-x(0)\rvert ds}\\
		\vspace{0.3cm}\displaystyle{\hspace{1.3cm}+2L\int_{-\frac{\tau}{2}}^{\theta^{*}-\tau}(-s)ds+2L\lVert h\rVert_{\infty}\left(\theta^{*}-\frac{\tau}{2}\right)}\\
		\displaystyle{\hspace{1cm}\leq LM\left(\frac{\tau}{2}\right)^{2}+3L\lVert h\rVert_{\infty}\tau +2L\int_{-\frac{\tau}{2}}^{\theta^{*}-\tau}\lvert \bar{y}(2s+\tau)-x(0)\rvert ds+L\theta^{*}\tau +2L\lVert h\rVert_{\infty}\theta^{*}.}
	\end{array}$$
	Thus, using \eqref{7} and \eqref{11} and the fact that $\frac{\tau}{2}\leq \theta^{*}$, it comes that

	$$\begin{array}{l}
		\vspace{0.3cm}\displaystyle{\lvert \bar{x}(0)+x(0)-2\tilde{x}(0)\rvert\leq 2LM\left(\frac{\tau}{2}\right)^{2}+6L\lVert h\rVert_{\infty}\frac{\tau}{2}}\\
		\vspace{0.3cm}\displaystyle{\hspace{4.1cm}+4L(1+2MC)\lVert h\rVert_{\infty}\left(\theta^{*}-\frac{\tau}{2}\right)+2L\theta^{*}\frac{\tau}{2}+2L\lVert h\rVert_{\infty}\theta^{*}}\\
		\vspace{0.3cm}\displaystyle{\hspace{2cm}\leq 2LM\left(\frac{\tau}{2}\right)^{2}+6L\lVert h\rVert_{\infty}\frac{\tau}{2}+4L(1+2MC)\lVert h\rVert_{\infty}\theta^{*}+2L\theta^{*}\frac{\tau}{2}+2L\lVert h\rVert_{\infty}\theta^{*}}\\
		\vspace{0.3cm}\displaystyle{\hspace{2cm}\leq 2LM\left(\theta^{*}\right)^{2}+6L\lVert h\rVert_{\infty}\theta^{*}+4L(1+2MC)\lVert h\rVert_{\infty}\theta^{*}+2L(\theta^{*})^{2}+2L\lVert h\rVert_{\infty}\theta^{*}}\\
		\vspace{0.3cm}\displaystyle{\hspace{2cm}=2L(1+M)(\theta^{*})^{2}+8L(1+MC)\lVert h\rVert_{\infty}\theta^{*}}\\
		\displaystyle{\hspace{2cm}\leq 8LC^{2}\lVert h\rVert_{\infty}^{2}+16LC(1+MC)\lVert h\rVert_{\infty}^{2}=8LC(C+2(1+MC))\lVert h\rVert_{\infty}^{2}.}
	\end{array}$$
	Hence, 
	\begin{equation}\label{13}
		\lvert \bar{x}(0)+x(0)-2\tilde{x}(0)\rvert\leq8LC(C+2(1+MC))\lVert h\rVert_{\infty}^{2}.
	\end{equation}
	Finally, assume that also $\theta^{*}>\tau$. Then,

	$$\begin{array}{l}
		\vspace{0.3cm}\displaystyle{\lvert \bar{x}(0)+x(0)-2\tilde{x}(0)\rvert\leq 2L\int_{-\tau}^{-\frac{\tau}{2}}\lvert \bar{y}(2s+\tau)-\tilde{y}(s)\rvert ds+2L\int_{-\frac{\tau}{2}}^{0}\lvert \bar{y}(2s+\tau)-\tilde{y}(s)\rvert ds}\\
		\vspace{0.3cm}\displaystyle{\hspace{4.2cm}+2L\int_{0}^{\theta^{*}-\tau}\lvert \bar{y}(2s+\tau)-\tilde{y}(s)\rvert ds}\\
		\vspace{0.3cm}\displaystyle{\hspace{2cm}=2L\int_{-\tau}^{-\frac{\tau}{2}}\lvert x(2s+\tau)+2h(2s+\tau)-x(s)-h(s)\rvert ds}\\
		\vspace{0.3cm}\displaystyle{\hspace{2.3cm}+ 2L\int_{-\frac{\tau}{2}}^{0}\lvert \bar{y}(2s+\tau)-x(s)-h(s)\rvert ds+2L\int_{0}^{\theta^{*}-\tau}\lvert \bar{y}(2s+\tau)-\tilde{y}(s)\rvert ds}\\
		\vspace{0.3cm}\displaystyle{\hspace{2cm}\leq2L\int_{-\tau}^{-\frac{\tau}{2}}\lvert x(2s+\tau)-x(s)\rvert ds+2L\int_{-\tau}^{-\frac{\tau}{2}}\lvert 2h(2s+\tau)-h(s)\rvert ds}\\
		\vspace{0.3cm}\displaystyle{\hspace{2.3cm}+2L\int_{-\frac{\tau}{2}}^{0}\lvert \bar{y}(2s+\tau)-x(0)\rvert ds+2L\int_{-\frac{\tau}{2}}^{0}\lvert x(0)-x(s)\rvert ds+2L \int_{-\frac{\tau}{2}}^{0}\lvert h(s)\rvert ds}\\
		\displaystyle{\hspace{2.3cm}+2L\int_{0}^{\theta^{*}-\tau}\lvert \bar{y}(2s+\tau)-\tilde{y}(s)\rvert ds.}		
	\end{array}$$
	Therefore, from \eqref{7}, \eqref{9}, \eqref{11} and from the fact that
	$x\in \mathcal{R}_{M}$, we can write

	$$\begin{array}{l}
		\vspace{0.3cm}\displaystyle{\lvert \bar{x}(0)+x(0)-2\tilde{x}(0)\rvert\leq8LC(C+2(1+MC))\lVert h\rVert_{\infty}^{2}+2L(1+6MC)\lVert h\rVert_{\infty}\theta^{*}}\\
		\vspace{0.3cm}\displaystyle{\hspace{2.5cm}\leq 8LC(C+2(1+MC))\lVert h\rVert_{\infty}^{2}+4LC(1+6MC)\lVert h\rVert_{\infty}^{2}}\\
		\displaystyle{\hspace{2.5cm}=4LC(C+2+2MC+2+12MC)\lVert h\rVert_{\infty}^{2},}
	\end{array}$$
	from which
	\begin{equation}\label{14}
		\lvert \bar{x}(0)+x(0)-2\tilde{x}(0)\rvert\leq4LC(C+4+14MC)\lVert h\rVert_{\infty}^{2}.
	\end{equation}
Combining \eqref{12}, \eqref{13} and \eqref{14}, we can conclude that 
\begin{equation}\label{15}
	\lvert \bar{x}(0)+x(0)-2\tilde{x}(0)\rvert\leq4LC(C+4+14MC)\lVert h\rVert_{\infty}^{2}.
\end{equation}
Now, we set

$$x_{1}=2\tilde{x}(0)-x(0),\quad x_{2}=x(0).$$	Then, $x_{2}\in\partial\mathcal{K}$ and also 

$$\frac{x_{1}+x_{2}}{2}=\tilde{x}(0)\in \partial\mathcal{K}.$$
Moreover, we have that $x_{1}\in \mathcal{K}_{r}$. Indeed, since $x(0)\in \partial\mathcal{K}$, \eqref{10} yields

$$\begin{array}{l}
	\vspace{0.3cm}\displaystyle{d_{\mathcal{K}}(2\tilde{x}(0)-x(0))\leq \lvert2\tilde{x}(0)-x(0)-x(0)\rvert}\\
	\vspace{0.3cm}\displaystyle{\hspace{1.5cm}=2\lvert\tilde{x}(0)-x(0)\rvert=2\lvert \tilde{y}(\theta^{*})-x(0)\rvert}\\
	\displaystyle{\hspace{1.5cm}\leq 2(1+2MC)\lVert h\rVert_{\infty}\leq 2(1+2MC)\rho <r.}
\end{array}$$
Thus, if $x_{1}\notin \mathring{\mathcal{K}}$, using \eqref{dk} and \eqref{10} we get

 $$\begin{array}{l}
	\vspace{0.3cm}\displaystyle{d_{\mathcal{K}}(2\tilde{x}(0)-x(0))=d_{\mathcal{K}}(2\tilde{x}(0)-x(0))+d_{\mathcal{K}}(x(0))-2d_{\mathcal{K}}(\tilde{x}(0))}\\
	\displaystyle{\hspace{1.5cm}\leq 4c\lvert\tilde{x}(0)-x(0)\rvert^{2}\leq 4c(1+2MC)^{2}\lVert h\rVert^{2}_{\infty}.}
\end{array}$$
This last inequality still holds whenever $x_{1}\in \mathring{\mathcal{K}}$ since, in this case, $d_{\mathcal{K}}(2\tilde{x}(0)-x(0))=0$.
Combining the above inequality with \eqref{14}, it comes that 

$$\begin{array}{l}
	\vspace{0.3cm}\displaystyle{d_{\mathcal{K}}(\bar{x}(0))\leq\lvert \bar{x}(0)-\pi(2\tilde{x}(0)-x(0))\rvert}\\
	\vspace{0.3cm}\displaystyle{\hspace{1.7cm}{\leq}\lvert \bar{x}(0)+x(0)-2\tilde{x}(0)\rvert+d_{\mathcal{K}}(2\tilde{x}(0)-x(0))}\\
	\displaystyle{\hspace{1.7cm}\leq 4LC(C+4+14MC)\lVert h\rVert_{\infty}^{2}+4c(1+2MC)^{2}\lVert h\rVert_{\infty}^{2}.}
\end{array}$$
So, setting 

$$\bar{k}:=4LC(7+16MC)+4c(1+2MC)^{2},$$
	we can write \begin{equation}\label{16}
		d_{\mathcal{K}}(\bar{x}(0))\leq\bar{k}\lVert h\rVert_{\infty}^{2}.
	\end{equation}
	Let us note that, being $x,x+h\in \mathcal{R}_{M}$, $x+2h\in \mathcal{R}_{3M}$. This implies that $\bar{x}=\bar{y}_{2\theta^{*}}(\cdot)=y_{2\theta^{*}}(\cdot;x+2h,\bar{u})\in \mathcal{R}_{3M}$. Therefore, $\bar{x}\in \mathcal{R}_{3M}$ with $\bar{x}(0)=\bar{y}(2\theta^{*})\in \mathcal{K}_{\delta}\setminus\mathcal{K}$ and \eqref{estimate} yields

	$$T(\bar{x})\leq Cd_{\mathcal{K}}(\bar{x}(0))).$$ 
	This last fact together with \eqref{theta*}, \eqref{pp} and \eqref{16} implies that

	$$\begin{array}{l}
		\vspace{0.3cm}\displaystyle{T(x+2h)\leq 2\theta^{*}+T(\bar{y}_{2\theta^{*}}(\cdot))}\\
		\vspace{0.3cm}\displaystyle{=2\theta^{*}+T(\bar{x})\leq  2\theta^{*}+Cd_{\mathcal{K}}(\bar{x}(0))}\\
		\vspace{0.3cm}\displaystyle{\leq2\theta^{*}+C\bar{k}\lVert h\rVert_{\infty}^{2}=2\theta^{*}+k\lVert h\rVert_{\infty}^{2}}\\
		\displaystyle{\leq 2(T(x+h)+\epsilon)+k\lVert h\rVert_{\infty}^{2},}
	\end{array}$$
showing that \eqref{8} is valid. Finally, letting $\epsilon\to0$ in \eqref{8}, we get	

$$T(x+2h)-2T(x+h)\leq k\lVert h\rVert_{\infty}^{2},$$
	and the proof is concluded.
\end{proof}
Now, we are able to prove Theorem \ref{semiconcavitythm}. We assume for simplicity that for any initial data $x\in \mathcal{R}$ there exists an optimal control. Indeed, if no such trajectory exists, one can use an approximation argument.
\begin{proof}[proof of Theorem \ref{semiconcavitythm}]
	Let $Q\subset\subset\mathcal{R}_{M}\setminus \mathcal{R}_{M}^{\mathcal{K}}$ be a convex set and take $x,h\in C([-\tau,0],\mathbb{R}^{n})$ such that $x,x+h,x-h\in Q$. Let $u$ be the optimal control associated with $x$, i.e. for which $\theta(x,u)=T(x)$. We set

 $$\bar{y}(t)=y(t;x,u),\quad y_{1}(t)=y(t;x-h,u), \quad y_{2}(t)=y(t;x+h,u).$$
	Then, by definition, $\bar{y}(T(x))\in \partial\mathcal{K}$.
\\Now, let $0\leq t\leq T(x)$. Then, we have that 

	$$\begin{array}{l}
		\vspace{0.3cm}\displaystyle{\lvert \bar{y}(t)-y_{1}(t)\rvert=\lvert y(t;x,u)-y(t;x-h,u)\rvert}\\
		\vspace{0.3cm}\displaystyle{\hspace{2.4cm}=\left\lvert x(0)+\int_{0}^{t}\bar{y}'(s)ds-x(0)+h(0)-\int_{0}^{t}y_{1}'(s)ds\right\rvert}\\
		\vspace{0.3cm}\displaystyle{\hspace{2.4cm}\leq \lvert h(0)\rvert+ \int_{0}^{t}\lvert f(\bar{y}(s-\tau),u)-f(y_{1}(s-\tau),u)\rvert ds}\\
		\vspace{0.3cm}\displaystyle{\hspace{2.4cm}\leq \lvert h(0)\rvert+L\int_{0}^{t}\lvert\bar{y}(s-\tau)-y_{1}(s-\tau)\rvert ds}\\
		\displaystyle{\hspace{2.4cm}=\lvert h(0)\rvert+L\int_{-\tau}^{t-\tau}\lvert\bar{y}(s)-y_{1}(s)\rvert ds.}
	\end{array}$$
	We distinguish two different situations. Assume first that $T(x)\leq \tau$. Then, $t\leq \tau$ and, for every $s\in [-\tau,t-\tau]$, it holds $s\leq 0$. Thus,	

$$\begin{array}{l}
		\vspace{0.3cm}\displaystyle{\lvert \bar{y}(t)-y_{1}(t)\rvert\leq \lvert h(0)\rvert+L\int_{-\tau}^{t-\tau}\lvert x(s)-x(s)+h(s)\rvert ds}\\
		\vspace{0.3cm}\displaystyle{\hspace{2.4cm}\leq \lVert h\rVert_{\infty}+L\lVert h\rVert_{\infty}t}\\
		\vspace{0.3cm}\displaystyle{\hspace{2.4cm}\leq\lVert h\rVert_{\infty}+L\lVert h\rVert_{\infty}T(x) }\\
		\displaystyle{\hspace{2.4cm}=(1+T(x))\lVert h\rVert_{\infty}.}
	\end{array}$$
	On the contrary, assume that $T(x)>\tau$. Then, if $t\in [0,\tau]$, for every $s\in [-\tau,t-\tau]$, we have that $s\leq 0$ and and

$$\lvert \bar{y}(t)-y_{1}(t)\rvert\leq \lvert h(0)\rvert+L\int_{-\tau}^{t-\tau}\lvert x(s)-x(s)+h(s)\rvert ds\leq \lVert h\rVert_{\infty}+L\lVert h\rVert_{\infty}t\leq (1+T(x))\lVert h\rVert_{\infty}.$$
	If $t>\tau$, it rather comes that
\begin{equation}\label{la_richiamo}	
\begin{array}{l}
		\vspace{0.3cm}\displaystyle{\lvert \bar{y}(t)-y_{1}(t)\rvert\leq\lVert h\rVert_{\infty}+L\int_{-\tau}^{0}\lvert\bar{y}(s)-y_{1}(s)\rvert ds+L\int_{0}^{t-\tau}\lvert\bar{y}(s)-y_{1}(s)\rvert ds}\\
		\vspace{0.3cm}\displaystyle{\hspace{2.4cm}=\lVert h\rVert_{\infty}+L\int_{-\tau}^{0}\lvert x(s)-x(s)+h(s)\rvert ds+L\int_{0}^{t-\tau}\lvert\bar{y}(s)-y_{1}(s)\rvert ds}\\
		\vspace{0.3cm}\displaystyle{\hspace{2.4cm}\leq \lVert h\rVert_{\infty}+L\lVert h\rVert_{\infty}\tau+L\int_{0}^{t-\tau}\lvert\bar{y}(s)-y_{1}(s)\rvert ds}\\
		\displaystyle{\hspace{2.4cm}\leq(1+LT(x))\lVert h\rVert_{\infty}+L\int_{0}^{ t}\lvert\bar{y}(s)-y_{1}(s)\rvert ds.}
	\end{array}
\end{equation}
	{Then, \eqref{la_richiamo} holds, as long as $\bar y$ and $y_1$ remain outside the target, for $t\in [0, T(x)].$} Thus, from Gronwall's inequality, we get

	$$\lvert \bar{y}(t)-y_{1}(t)\rvert\leq(1+LT(x))\lVert h\rVert_{\infty}
e^{Lt}\leq (1+LT(x))e^{LT(x)}\lVert h\rVert_{\infty}, \quad {\ t\le T(x).}$$
	We have so proved that, in both cases, \begin{equation}\label{1t}
		\lvert \bar{y}(t)-y_{1}(t)\rvert\leq k_{1}\lVert h\rVert_{\infty}, \quad  t\leq T(x),
	\end{equation}
	for some positive constant $k_{1}$, depending only on $Q$.
	Arguing in the same way, it also holds that
	\begin{equation}\label{2t}
		\lvert \bar{y}(t)-y_{2}(t)\rvert\leq k_{1}\lVert h\rVert_{\infty}, \quad t\leq T(x).
	\end{equation}
	{Now, arguing as in the proof of Lemma 7.1.2 in \cite{Cannarsa_libro}, from assumption $(H_{3})$ there exists a constant $k_{2}>0$, depending only on $Q$, such that}
	\begin{equation}\label{3t}
		{\left\lvert f(y_{1}(t),u)+f(y_{2}(t),u)-2f\left(\frac{y_{1}(t)+y_{2}(t)}{2},u\right)\right\rvert\leq k_{2}\lvert y_{1}(t)-y_{2}(t)\rvert^{2},}
	\end{equation}
{for every $t\in [0,T(x)]$.} Therefore, for $t\in [0,T(x)]$, we have that 

	$$\begin{array}{l}
		\vspace{0.3cm}\displaystyle{\lvert y_{1}(t)+y_{2}(t)-2\bar{y}(t)\rvert=\lvert y(t;x-h,u)+y(t;x+h,u)-2y(t;x,u)\rvert}\\
		\vspace{0.3cm}\displaystyle{\hspace{1.7cm}=\left\lvert x(0)-h(0)+\int_{0}^{t}y_{1}'(s)ds +x(0)+h(0)+\int_{0}^{t}y_{2}'(s)ds-2x(0)-2\int_{0}^{t}\bar{y}'(s) ds\right\rvert}\\
		\vspace{0.3cm}\displaystyle{\hspace{1.7cm}\leq \int_{0}^{t}\lvert y_{1}'(s)+y_{2}'(s)-2\bar{y}'(s) \rvert ds}\\
		\vspace{0.3cm}\displaystyle{\hspace{1.7cm}= \int_{0}^{t}\lvert f(y_{1}(s-\tau),u)+f(y_{2}(s-\tau),u)-2f(\bar{y}(s-\tau),u)\rvert ds}\\
		\displaystyle{\hspace{1.7cm}=\int_{-\tau}^{t-\tau}\lvert f(y_{1}(s),u)+f(y_{2}(s),u)-2f(\bar{y}(s),u)\rvert ds.}		
	\end{array}$$
Thus, $(H_{1})$ and \eqref{3t} imply that 

	$$\begin{array}{l}
		\vspace{0.3cm}\displaystyle{\lvert y_{1}(t)+y_{2}(t)-2\bar{y}(t)\rvert\leq \int_{-\tau}^{t-\tau}	\left\lvert f(y_{1}(s),u)+f(y_{2}(s),u)-2f\left(\frac{y_{1}(s)+y_{2}(s)}{2},u\right)\right\rvert ds}\\
		\vspace{0.3cm}\displaystyle{\hspace{4.3cm}+2\int_{-\tau}^{t-\tau}\left\lvert f\left(\frac{y_{1}(s)+y_{2}(s)}{2},u\right)-f(\bar{y}(s),u)\right\rvert ds.}\\
		\vspace{0.3cm}\displaystyle{\hspace{3.5cm}\leq k_{2}\int_{-\tau}^{t-\tau}\lvert y_{1}(t)-y_{2}(t)\rvert^{2}ds+2L\int_{-\tau}^{t-\tau}\left\lvert \frac{y_{1}(s)+y_{2}(s)}{2}-\bar{y}(s)\right\rvert ds}\\
		\displaystyle{\hspace{3.5cm}=k_{2}\int_{-\tau}^{t-\tau}\lvert y_{1}(s)-y_{2}(s)\rvert^{2}ds+L\int_{-\tau}^{t-\tau}\left\lvert y_{1}(s)+y_{2}(s)-2\bar{y}(s)\right\rvert ds.}
	\end{array}$$
	Now, we distinguish two situations. Assume first that if $t\leq\tau$. Then,
	
$$\begin{array}{l}
		\vspace{0.3cm}\displaystyle{\lvert y_{1}(t)+y_{2}(t)-2\bar{y}(t)\rvert\leq k_{2}\int_{-\tau}^{t-\tau}\lvert x(s)-h(s)-x(s)-h(s)\rvert^{2}ds}\\
		\vspace{0.3cm}\displaystyle{\hspace{4.4cm}+L\int_{-\tau}^{t-\tau}\left\lvert x(s)-h(s)+x(s)+h(s)-2x(s)\right\rvert ds}\\
		\displaystyle{\hspace{4cm}=4k_{2}\int_{-\tau}^{t-\tau}\lvert h(s)\rvert^{2}\leq 4k_{2}\lVert h\rVert_{\infty}^{2}t,}
	\end{array}$$
	and 
	\begin{equation}\label{17}
		\lvert y_{1}(t)+y_{2}(t)-2\bar{y}(t)\rvert\leq 4k_{2}T(x)\lVert h\rVert_{\infty}^{2}.
	\end{equation}
	On the other hand, assume that $t>\tau$. Then, in this case,

	$$\begin{array}{l}
		\vspace{0.3cm}\displaystyle{\lvert y_{1}(t)+y_{2}(t)-2\bar{y}(t)\rvert\leq { k_2} \int_{-\tau}^{0}\lvert y_{1}(s)-y_{2}(s)\rvert^{2}ds+L\int_{-\tau}^{0}\left\lvert y_{1}(s)+y_{2}(s)-2\bar{y}(s)\right\rvert ds}\\
		\vspace{0.3cm}\displaystyle{\hspace{2cm}+\int_{0}^{t-\tau}\lvert y_{1}(s)-y_{2}(s)\rvert^{2}ds+L\int_{0}^{t-\tau}\left\lvert y_{1}(s)+y_{2}(s)-2\bar{y}(s)\right\rvert ds}\\
		\vspace{0.3cm}\displaystyle{\hspace{1cm}=4k_{2}\int_{-\tau}^{0}\lvert h(s)\rvert^{2} ds+L\int_{-\tau}^{0}\lvert x(s)-h(s)+x(s)+h(s)-2x(s)\rvert ds}\\
		\vspace{0.3cm}\displaystyle{\hspace{2cm}+\int_{0}^{t-\tau}\lvert y_{1}(s)-y_{2}(s)\rvert^{2}ds+L\int_{0}^{t-\tau}\left\lvert y_{1}(s)+y_{2}(s)-2\bar{y}(s)\right\rvert ds}\\
		\displaystyle{\hspace{1cm}\leq 4k_{2}\lVert h\rVert_{\infty}^{2}\tau+\int_{0}^{t-\tau}\lvert y_{1}(s)-y_{2}(s)\rvert^{2}ds+L\int_{0}^{t-\tau}\left\lvert y_{1}(s)+y_{2}(s)-2\bar{y}(s)\right\rvert ds.}
	\end{array}$$
	Now, using \eqref{1t} and \eqref{2t}, for every $s\in [0,t-\tau]$, 

	$$\lvert y_{1}(s)-y_{2}(s)\rvert\leq \lvert y_{1}(s)-\bar{y}(s)\rvert+\lvert \bar{y}(s)-y_{2}(s)\rvert\leq 2k_{1}\lVert h\rVert_{\infty},$$
	from which 

$$\lvert y_{1}(s)-y_{2}(s)\rvert^{2}\leq 4k_{1}^{2}\lVert h\rVert_{\infty}^{2}.$$
	As a consequence, we get 
\begin{equation}\label{da_richiamare}	
\begin{array}{l}
		\vspace{0.3cm}\displaystyle{\lvert y_{1}(t)+y_{2}(t)-2\bar{y}(t)\rvert\leq 4k_{2}\lVert h\rVert_{\infty}^{2}\tau+4k_{1}^{2}\lVert h\rVert_{\infty}^{2}(t-\tau)}\\
		\vspace{0.3cm}\displaystyle{\hspace{4.5cm}+L\int_{0}^{t-\tau}\left\lvert y_{1}(s)+y_{2}(s)-2\bar{y}(s)\right\rvert ds}\\
		\vspace{0.3cm}\displaystyle{\hspace{4cm}\leq 4(k_{2}+k_{1}^{2})T(x)\lVert h\rVert_{\infty}^{2}+L\int_{0}^{t-\tau}\left\lvert y_{1}(s)+y_{2}(s)-2\bar{y}(s)\right\rvert ds}\\
		\displaystyle{\hspace{4cm}\leq 4(k_{2}+k_{1}^{2})T(x)\lVert h\rVert_{\infty}^{2}+L\int_{0}^{t}\left\lvert y_{1}(s)+y_{2}(s)-2\bar{y}(s)\right\rvert ds.}
	\end{array}
\end{equation}
Then, \eqref{da_richiamare} holds for any $t\le T(x).$
	Thus, Gronwall's inequality yields

	$$\lvert y_{1}(t)+y_{2}(t)-2\bar{y}(t)\rvert\leq e^{LT(x)}4(k_{2}+k_{1}^{2})T(x)\lVert h\rVert_{\infty}^{2}, \quad t\le T(x).$$
	This last fact together with \eqref{17} implies that 
	\begin{equation}\label{disug}
		\lvert y_{1}(t)+y_{2}(t)-2\bar{y}(t)\rvert\leq k_{3}\lVert h \rVert_{\infty}^{2} ,\quad  t\leq T(x),
	\end{equation}
	for a positive constant $k_{3}$, dependent only of $Q$. Now, we can consider three cases.

(Case I) Suppose that one of the two paths $y_{1}$ or $y_{2}$, namely $y_{1}$, reaches $\mathcal{K}$ at a time $t^{*}<T(x)$, i.e. $t^{*}=\theta(x-h,u)<T(x)$. We define the functions $x_{1}$, $x_{2}$, $\bar{x}$ in $C([-\tau,0];\mathbb{R}^{n})$ as follows
	
$$x_{1}(\cdot)=y_{1_{t^{*}}}(\cdot),\quad x_{2}(\cdot)=y_{2_{t^{*}}}(\cdot),\quad \bar{x}(\cdot)=\bar{y}_{t^{*}}(\cdot).$$
	We can immediately notice that $x_{1}(0)=y_{1}(t^{*})\in \partial\mathcal{K}$, being $\theta(x-h,u)=t^{*}$. Also, by definition of $T$,
	\begin{equation}\label{dpp1}
		T(x-h)\leq t^{*}.
	\end{equation}
	Moreover, using the dynamic programming principle (Proposition \ref{din}) and since $u$ is an optimal control for $x$, we can write
	\begin{equation}\label{dpp2}
		T(x+h)\leq t^{*}+T(x_{2}),\quad T(x)=t^{*}+T(\bar{x}).
	\end{equation}
	Now, we want to apply Lemma \ref{auxlem}. First of all, we have that $x_{1}(0)\in \partial\mathcal{K}$ and that
$x_{1}=y_{1_{t^{*}}}(\cdot)=y_{t^{*}}(\cdot;x-h,u)\in \mathcal{R}_{M}$. In addition, $\bar{x}-x_{1}\in C([-\tau,0];\mathbb{R}^{n})$ satisfies: $(x_{1}+\bar{x}-x_{1})(0)=\bar{x}(0)=\bar{y}_{t^{*}}(0)=y(t^{*};x,u)\notin\mathcal{K}$, since from our assumptions $t^{*}<T(x)$; $x_{1}+\bar{x}-x_{1}=\bar{x}=\bar{y}_{t^{*}}(\cdot)=y_{t^{*}}(\cdot;x,u)\in \mathcal{R}_{M}.$ 
Therefore, from Lemma \ref{auxlem}, for $\lVert \bar{x}-x_{1}\rVert_{\infty}$ small it comes that 
	\begin{equation}\label{4t}
		T(2\bar{x}-x_{1})-2T(\bar{x})\leq k_{4}\lVert \bar{x}-x_{1}\rVert_{\infty}^{2},
	\end{equation}
for a suitable positive constant $k_{4}$. Moreover, for every $s\in [-\tau,0]$, we have
	
$$\lvert \bar{x}(s)-x_{1}(s)\rvert=\lvert\bar{y}(t^{*}+s)-y_{1}(t^{*}+s)\rvert.$$
Thus, if $t^{*}+s>0$, from \eqref{1t} with $t^{*}+s\in [0,T(x)]  $ we get

	$$\lvert \bar{x}(s)-x_{1}(s)\rvert\leq k_{1}\lVert h\rVert_{\infty}.$$
On the other hand, if $t^{*}+s\leq0$, it rather holds 

$$\begin{array}{l}
	\vspace{0.3cm}\displaystyle{\hspace{0.5cm}\lvert \bar{x}(s)-x_{1}(s)\rvert=\lvert y(t^{*}+s;x,u)-y(t^{*}+s;x-h,u)\rvert}\\
	\displaystyle{=\lvert x(t^{*}+s)-x(t^{*}+s)+h(t^{*}+s)\rvert=\lvert h(t^{*}+s)\rvert\leq \lVert h\rVert_{\infty}}.
\end{array}$$
\\So, setting $k_{5}:=\max\{k_{1},1\}$, it comes that

$$\lvert \bar{x}(s)-x_{1}(s)\rvert\leq k_{5}\lVert h\rVert_{\infty}, $$
for all $s\in [-\tau,0]$. Therefore,

	$$\lVert\bar{x}-x_{1}\rVert_{\infty}\leq k_{5}\lVert h\rVert_{\infty}. $$
and \eqref{4t} becomes 
\begin{equation}\label{5t}
		T(2\bar{x}-x_{1})-2T(\bar{x})\leq k_{4}k_{5}^{2}\lVert h\rVert_{\infty}^{2}.
	\end{equation}
Furthermore, since $T$ is locally Lipschitz continuous from Theorem \ref{lip}, we have that
	\begin{equation}\label{6t}
		\lvert T(2\bar{x}-x_{1})-T(x_{2})\rvert\leq k_{6} \lVert 2\bar{x}-x_{1}-x_{2}\rVert_{\infty},
	\end{equation}
for some positive constant $k_{6}$.	
Let us note that, for every $s\in [-\tau,0]$, 

$$\lvert2\bar{x}(s)-x_{1}(s)-x_{2}(s)\rvert=\lvert2\bar{y}(t^{*}+s)-y_{1}(t^{*}+s)-y_{2}(t^{*}+s)\rvert.$$
Then, if $t^{*}+s>0$, from \eqref{disug} it follows that 

$$\lvert2\bar{x}(s)-x_{1}(s)-x_{2}(s)\rvert\leq k_{3}\lVert h\rVert_{\infty}^{2}.$$
On the contrary, if $t^{*}+s\leq0$, it turns out that

$$\begin{array}{l}
	\vspace{0.3cm}\displaystyle{\lvert2\bar{x}(s)-x_{1}(s)-x_{2}(s)\rvert=\lvert2y(t^{*}+s;x,u)-y(t^{*}+s;x-h,u)-y(t^{*}+s;x+h,u)\rvert}\\
	\displaystyle{\hspace{0.7cm}=\lvert 2x(t^{*}+s)-x(t^{*}+s)+h(t^{*}+s)-x(t^{*}+s)-h(t^{*}+s)\rvert=0\leq k_{3}\lVert h\rVert_{\infty}^{2}.}
\end{array}$$
Thus, we can write

	$$\lvert2\bar{x}(s)-x_{1}(s)-x_{2}(s)\rvert\leq  k_{3}\lVert h\rVert_{\infty}^{2},$$ 
for every $s\in [-\tau,0]$. Hence, 

$$\lVert 2\bar{x}-x_{1}-x_{2}\rVert_{\infty}\leq k_{3}\lVert h\rVert_{\infty}^{2},$$
from which \eqref{6t} becomes \begin{equation}\label{7t}
		\lvert T(2\bar{x}-x_{1})-2T(\bar{x})\rvert\leq k_{6}k_{3}\lVert h\rVert_{\infty}^{2}.
\end{equation}
As a consequece, using \eqref{dpp1}, \eqref{dpp2}, \eqref{5t} and \eqref{7t}, we finally get

	$$\begin{array}{l}
		\vspace{0.3cm}\displaystyle{T(x+h)+T(x-h)-2T(x)\leq t^{*}+T(x_{2})+t^{*}-2t^{*}-2T(\bar{x})}\\
		\vspace{0.3cm}\displaystyle{\hspace{2cm}=T(x_{2})-T(2\bar{x}-x_{1})+T(2\bar{x}-x_{1})-2T(\bar{x})}\\
		\displaystyle{\hspace{2cm}\leq (k_{4}k_{5}^{2}+k_{6}k_{3})\lVert h\rVert_{\infty}^{2}=k_{7}\lVert h\rVert_{\infty}^{2}.}
	\end{array}$$

(Case II) Suppose that neither $y_{1}$ nor $y_{2}$ reach the target before $\bar{y}$ and that \begin{equation}\label{app}
		\frac{y_{1}(T(x))+y_{2}(T(x))}{2}\in\mathcal{K}.
	\end{equation}
Then, there exists $t^{*}\leq T(x)$ such that 

$$\frac{y_{1}(t^{*})+y_{2}(t^{*})}{2}\in\partial\mathcal{K}.$$
Let us note that, being $t^{*}\leq T(x)$,  $y_{1}(t^{*}),\,y_{2}(t^{*})\notin \mathcal{K}$. Thus, from \eqref{neccond} there exists a positive constant $k_{8}$ such that \begin{equation}\label{8t}
		d_{\mathcal{K}}(y_{1}(t^{*}))+d_{\mathcal{K}}(y_{2}(t^{*}))-2d_{\mathcal{K}}\left(\frac{y_{1}(t^{*})+y_{2}(t^{*})}{2}\right)\leq k_{8}\lvert y_{2}(t^{*})-y_{1}(t^{*})\rvert^{2}.
	\end{equation}
Now, using \eqref{1t} and \eqref{2t}, it holds that

	$$\lvert y_{2}(t^{*})-y_{1}(t^{*})\rvert\leq \lvert y_{2}(t^{*})-\bar{y}(t^{*})\rvert+\lvert \bar{y}(t^{*})-y_{1}(t^{*})\rvert\leq 2k_{1}\lVert h\rVert_{\infty},$$
from which 

$$\lvert y_{2}(t^{*})-y_{1}(t^{*})\rvert^{2}\leq 4k_{1}^{2}\lVert h\rVert_{\infty}^{2}.$$
Therefore, combining this last fact with \eqref{app} and \eqref{8t}, it follows that
	\begin{equation}\label{dk1}
		\begin{split}
			d_{\mathcal{K}}(y_{1}(t^{*}))+d_{\mathcal{K}}(y_{2}(t^{*}))&=d_{\mathcal{K}}(y_{1}(t^{*}))+d_{\mathcal{K}}(y_{2}(t^{*}))-2d_{\mathcal{K}}\left(\frac{y_{1}(t^{*})+y_{2}(t^{*})}{2}\right)\\&\leq 4k_{8}k_{1}^{2}\lVert h\rVert_{\infty}^{2}.
		\end{split}
	\end{equation}
	Also, the dynamic programming principle (Proposition \ref{din}) yields \begin{equation}\label{dpp3}
		T(x-h)\leq t^{*}+T(x_{1}),\quad T(x+h)\leq t^{*}+T(x_{2}),
	\end{equation}
	where $x_{1},x_{2}$ are the functions in $C([-\tau,0];\mathbb{R}^{n})$ given by

	$$x_{1}(\cdot)=y_{1_{t^{*}}}(\cdot),\quad x_{2}(\cdot)=y_{2_{t^{*}}}(\cdot).$$
	Now, from \eqref{estimate} it comes that \begin{equation}\label{9t}
		T(x_{1})\leq Cd_{\mathcal{K}}(x_{1}(0)),\quad T(x_{2})\leq Cd_{\mathcal{K}}(x_{2}(0)).
	\end{equation}
	As a consequence, putting together \eqref{dk1}, \eqref{dpp3} and \eqref{9t}, we can conclude that

	$$\begin{array}{l}
		\vspace{0.3cm}\displaystyle{T(x+h)+T(x-h)-2T(x)\leq t^{*}+T(x_{1})+t^{*}+T(x_{2})-2T(x)}\\
		\vspace{0.3cm}\displaystyle{\hspace{5.5cm}\leq 2T(x)+T(x_{1})+T(x_{2})-2T(x)}\\
		\vspace{0.3cm}\displaystyle{\hspace{5.5cm}\leq C(d_{\mathcal{K}}(x_{1}(0))+d_{\mathcal{K}}(x_{2}(0)))}\\
		\vspace{0.3cm}\displaystyle{\hspace{5.5cm}=d_{\mathcal{K}}(y_{1}(t^{*}))+d_{\mathcal{K}}(y_{2}(t^{*}))}\\
		\vspace{0.3cm}\displaystyle{\hspace{5.5cm}\leq 4k_{8}k_{1}^{2}\lVert h\rVert_{\infty}^{2}}\\
		\displaystyle{\hspace{5.5cm}=k_{9}\lVert h\rVert_{\infty}^{2}.}
	\end{array}$$

(Case III) Suppose that neither $y_{1}$ nor $y_{2}$ reach the target before $\bar{y}$ and that

 $$\frac{y_{1}(T(x))+y_{2}(T(x))}{2}\notin\mathcal{K}.$$ 
	Let $x_{1}$, $x_{2}$, $\bar{x}$ be the functions in $C([-\tau,0];\mathbb{R}^{n})$ defined as follows

	$$x_{1}(\cdot)=y_{1_{T(x)}}(\cdot),\quad x_{2}(\cdot)=y_{2_{T(x)}}(\cdot),\quad\bar{x}(\cdot)=\bar{y}_{T(x)}(\cdot).$$
	Then, $x_{1}(0),\,x_{2}(0), \frac{x_{1}(0)+x_{2}(0)}{2}\notin\mathcal{K}$. Moreover, from \eqref{neccond}, \eqref{1t} and \eqref{2t}, we can write

	$$\begin{array}{l}
		\vspace{0.3cm}\displaystyle{d_{\mathcal{K}}(x_{1}(0))+d_{\mathcal{K}}(x_{2}(0))-2d_{\mathcal{K}}\left(\frac{x_{1}(0)+x_{2}(0)}{2}\right)\leq k_{9}\lvert x_{1}(0)-x_{2}(0)\rvert^{2}}\\
		\displaystyle{\hspace{2cm}=k_{8}\lvert y_{1}(T(x))-y_{2}(T(x))\rvert^{2}\leq 4k_{8}k_{1}^{2}\lVert h\rVert_{\infty}^{2}.}
	\end{array}$$
	Therefore, since $\bar{x}(0)=\bar{y}(T(x))=y(T(x);x,u)\in\mathcal{K}$, using \eqref{disug} we get 

	$$\begin{array}{l}
		\vspace{0.3cm}\displaystyle{d_{\mathcal{K}}(x_{1}(0))+d_{\mathcal{K}}(x_{2}(0))}\\
		\vspace{0.3cm}\displaystyle{\hspace{1.5cm}=d_{\mathcal{K}}(x_{1}(0))+d_{\mathcal{K}}(x_{2}(0))-2d_{\mathcal{K}}\left(\frac{x_{1}(0)+x_{2}(0)}{2}\right)+2d_{\mathcal{K}}\left(\frac{x_{1}(0)+x_{2}(0)}{2}\right)}\\
		\vspace{0.3cm}\displaystyle{\hspace{1.5cm}\leq 4k_{8}k_{1}^{2}\lVert h\rVert_{\infty}^{2}+2\left\lvert \frac{x_{1}(0)+x_{2}(0)}{2}-\bar{x}(0)\right\rvert}\\
		\vspace{0.3cm}\displaystyle{\hspace{1.5cm}=4k_{8}k_{1}^{2}\lVert h\rVert_{\infty}^{2}+\lvert x_{1}(0)+x_{2}(0)-2\bar{x}(0)\rvert}\\
		\vspace{0.3cm}\displaystyle{\hspace{1.5cm}=4k_{8}k_{1}^{2}\lVert h\rVert_{\infty}^{2}+\lvert y_{1}(T(x))+y_{2}(T(x))-2\bar{y}(T(x))\rvert}\\
		\displaystyle{\hspace{1.5cm}\leq 4k_{8}k_{1}^{2}\lVert h\rVert_{\infty}^{2}+k_{3}\lVert h\rVert_{\infty}^{2}=(4k_{8}k_{1}^{2}+k_{3})\lVert h\rVert_{\infty}^{2}.}
	\end{array}$$
	Finally, from Proposition \ref{din} and from \eqref{estimate},

	$$\begin{array}{l}
		\vspace{0.3cm}\displaystyle{T(x+h)+T(x-h)-2T(x)\leq T(x)+T(x_{1})+T(x)+T(x_{2})-2T(x)}\\
		\vspace{0.3cm}\displaystyle{\hspace{5.0cm}\leq C(d_{\mathcal{K}}(x_{1}(0))+d_{\mathcal{K}}(x_{2}(0)))}\\
		\displaystyle{\hspace{5.0cm}\leq C(4k_{8}k_{1}^{2}+k_{3})\lVert h\rVert_{\infty}^{2}=k_{10}\lVert h\rVert_{\infty}^{2},}
	\end{array}$$
	which concludes our proof.
\end{proof}
\section{Conclusions}
In this paper, we investigated the regularity properties of the value function associated to a minimum time problem for a delayed control system.
 
Time delays in differential equations are very useful for the description of several phenomena. There are situations in which the dynamics does not depend on the present state of the trajectory but is rather influenced by what has happened in some previous instants. For instance, in stock exchange, one might have to look at the past market trend in order to make the right investment. Other situations in which it is more convenient to look at the past history of the dynamics can be found in epidemiological, biological, social, and economic models.

First, we established  a Lipschitz continuity result and then, using such a property,  a semiconcavity result. 

Semiconcavity in the classical undelayed case has many applications in optimal control theory.  It can be useful to derive optimality conditions (see e.g. \cite{CFS, CPS}). The analysis of the regularity properties of the minimum time function is also important since it is connected with the study of related Hamilton-Jacobi-Bellman equations. 

The Hamilton-Jacobi theory in infinite dimension has been developed by \cite{Crandall1, Crandall2, Crandall3}. Also, for optimal control problems involving time delays and with finite horizon, the Hamilton-Jacobi-Bellman equations have been investigated by \cite{Barron, Boccia, Bonalli, Luk1, Luk2, Luk3, Plaksin1, Plaksin2, Soner, Vinter, Zhu}. It is, then, interesting to discuss the Hamilton-Jacobi theory related to the minimum time problem for control systems presenting time delays in the state space and to deepen the application of the semiconcavity also in the delayed setting. We leave this for further research.

\bigskip
\noindent {\bf Acknowledgements.} The authors are members of  {\it Gruppo Nazionale per l'Analisi Ma\-te\-matica, la Probabilit\`a e le loro Applicazioni (GNAMPA)} of the Istituto Nazionale di Alta Matematica (INdAM). They are also  members of {\it UMI ``CliMath"}.
They  are partially supported by PRIN 2022  (2022238YY5) {\it Optimal control problems: analysis,
approximation and applications} and by INdAM GNAMPA Project {\it ``Modelli alle derivate parziali per interazioni multiagente non 
simmetriche"}(CUP E53C23001670001). 
C. Pignotti is also partially supported by
PRIN-PNRR 2022 (P20225SP98) {\it Some mathematical approaches to climate change and its impacts}.

\end{document}